\documentclass{article}
\usepackage[T1]{fontenc}    
\usepackage[english]{babel}
\usepackage{latexsym}
\usepackage{amssymb}
\usepackage{amsmath}
\usepackage{amsthm}
\usepackage{graphicx}
\usepackage{hyperref}
\usepackage{color}
\usepackage{colortbl}
\usepackage{epsfig}

\newtheorem{de}{Définition}[section] 
\newtheorem{prop}{Proposition}[section] 

\begin{document}


\title{Slow Invariant Manifolds as\\
Curvature of the Flow of Dynamical Systems}

\maketitle

\begin{center}

\author{JEAN-MARC GINOUX, BRUNO ROSSETTO, LEON O. CHUA$^{\dag }$\\
\textit{Laboratoire PROTEE,}
\\
\textit{I.U.T. of Toulon, Université du Sud, }
\\
\textit{B.P. 20132,
83957, LA GARDE Cedex, France}
\\

$^{\dag }$\textit{EECS Department University of California,
Berkeley} \\

\textit{253 Cory Hall {\#}1770, Berkeley, CA 94720-1770}

\textcolor{blue}{ginoux@univ-tln.fr},
\textcolor{blue}{rossetto@univ-tln.fr}
\textcolor{blue}{chua@eecs.berkeley.edu}}

\end{center}

\begin{center}
website: \href{http://ginoux.univ-tln.fr}{http://ginoux.univ-tln.fr},
\href{http://rossetto.univ-tln.fr}{http://rossetto.univ-tln.fr},
\href{http://www.eecs.berkeley.edu/~chua}{http://www.eecs.berkeley.edu/$\sim
$chua}
\end{center}

\begin{abstract}

Considering trajectory curves, integral of n-dimensional dynamical
systems, within the framework of Differential Geometry as curves in
Euclidean n-space it will be established in this article that the
curvature of the flow, i.e., the curvature of the trajectory curves
of any n-dimensional dynamical system directly provides its slow
manifold analytical equation the invariance of which will be then
proved according to Darboux theory. Thus, it will be stated that the
flow curvature method, which uses neither eigenvectors nor
asymptotic expansions but only involves time derivatives of the
velocity vector field, constitutes a general method simplifying and
improving the slow invariant manifold analytical equation
determination of high-dimensional dynamical systems. Moreover, it
will be shown that this method generalizes the Tangent Linear System
Approximation and encompasses the so-called Geometric Singular
Perturbation Theory. Then, slow invariant manifolds analytical
equation of paradigmatic Chua's piecewise linear and cubic models of
dimensions three, four, and five will be provided as tutorial
examples exemplifying this method as well as those of
high-dimensional dynamical systems.

\vspace{0.1in}


\end{abstract}

\section{Introduction}

Dynamical systems consisting of \textit{nonlinear} differential
equations are generally not integrable. In his famous memoirs:
\textit{Sur les courbes d\'{e}finies par une \'{e}quation
diff\'{e}rentielle}, Poincar\'{e} [1881-1886] faced to this problem
proposed to study \textit{trajectory curves} properties in the
\textit{phase space}.\\

``\ldots any differential equation can be written as:

\[
\frac{dx_1 }{dt} = X_1 ,
\quad
\frac{dx_2 }{dt} = X_2 , \quad \ldots ,
\quad
\frac{dx_n }{dt} = X_n
\]

where $X$ are integer polynomials.

If $t$ is considered as the time, these equations will define the motion of
a variable point in a space of dimension $n$.''

\begin{flushright}
-- Poincar\'{e} [1885, p. 168] --
\end{flushright}

Let's consider the following system of differential equations
defined in a compact E included in $\mathbb{R}$ as:

\begin{equation}
\label{eq1}
\frac{d\vec {X}}{dt} = \overrightarrow \Im \left( \vec {X} \right)
\end{equation}

with

\[
\vec {X} = \left[ {x_1 ,x_2 ,...,x_n } \right]^t \in E \subset \mathbb{R}^n
\]

and

\[
\overrightarrow \Im \left( \vec {X} \right) = \left[ {f_1 \left( \vec {X}
\right),f_2 \left( \vec {X} \right),...,f_n \left( \vec {X} \right)}
\right]^t \in E \subset \mathbb{R}^n
\]

\vspace{0.2in}

The vector $\overrightarrow \Im \left( \vec {X} \right)$ defines a velocity
vector field in E whose components $f_i $ which are supposed to be
continuous and infinitely differentiable with respect to all $x_i $ and $t$,
i.e., are $C^\infty $ functions in E and with values included in
$\mathbb{R}$, satisfy the assumptions of the Cauchy-Lipschitz theorem. For
more details, see for example Coddington \textit{et al.} [1955]. A solution of this system
is a \textit{trajectory curve} $\vec {X}\left( t \right)$ tangent\footnote{ Except at the \textit{fixed points}.} to
$\overrightarrow \Im $ whose values define the \textit{states} of the \textit{dynamical system} described by the Eq.
(\ref{eq1}). Since none of the components $f_i $ of the velocity vector field
depends here explicitly on time, the system is said to be \textit{autonomous}.

Thus, \textit{trajectory curves} integral of dynamical systems
(\ref{eq1}) regarded as $n$-dimensional \textit{curves}, possess
local metrics properties, namely \textit{curvatures }which can be
analytically\footnote{ Since only time derivatives of the
\textit{trajectory curves} are involved in the \textit{curvature}
formulas. } deduced from the so-called Fr\'{e}net formulas recalled
in the next section. For low dimensions two and three the concept of
\textit{curvatures} may be simply exemplified. A
three-dimensional\footnote{ A two-dimensional curve, i.e., a plane
\textit{curve} has a \textit{torsion} vanishing identically.} curve
for example has two \textit{curvatures: curvature} and
\textit{torsion} which are also known as \textit{first} and
\textit{second curvature}. \textit{Curvature}\footnote{The notion of \textit{curvature} of a plane curve first appears in the work of Apollonius of Perga.} measures, so to speak,
the deviation of the curve from a straight line in the neighbourhood
of any of its points. While the \textit{torsion}\footnote{The name \textit{torsion} is due to L.I. Vallée, \textit{Traité de Géométrie Descriptive}.} measures, roughly
speaking, the magnitude and sense of deviation of the curve from the
\textit{osculating plane}\footnote{ The \textit{osculating plane} is
defined as the plane spanned by the instantaneous velocity and
acceleration vectors.} in the neighbourhood of the corresponding
point of the curve, or, in other words, the rate of change of the
\textit{osculating plane}. Physically, a three-dimensional curve may
be obtained from a straight line by bending (\textit{curvature}) and
twisting (torsion). For high dimensions greater than three, say $n$,
a $n$-dimensional curve has $\left( {n - 1} \right)$
\textit{curvatures} which may be computed while using the
Gram-Schmidt orthogonalization process [Gluck, 1966]. This
procedure, presented in Appendix, also enables to define the
Fr\'{e}net formulas for a $n$-dimensional curve.

In a recent publication [Ginoux \textit{et al.}, 2006] it has been established that the
location of the point where the \textit{curvature of the flow}, i.e., the \textit{curvature} of the \textit{trajectory curves }integral of any
\textit{slow-fast dynamical systems} of low dimensions two and three vanishes directly provides the \textit{slow} \textit{invariant } \textit{manifold}
analytical equation associated to such dynamical systems. So, in this work
the new approach proposed by Ginoux \textit{et al.} [2006] is generalized to
high-dimensional dynamical systems.

The main result of this work presented in the first section
establishes that \textit{curvature of the flow}, i.e.,
\textit{curvature} of \textit{trajectory curves} of any
$n$-dimensional dynamical system directly provides its \textit{slow
manifold} analytical equation the \textit{invariance} of which is
proved according to \textit{Darboux Theorem}.

Then, Chua's piecewise linear models of dimensions three, four and
five are used in the second section to exemplify this result. Indeed
it has been already established [Chua 1986] that such
\textit{slow-fast dynamical systems} exhibit \textit{trajectory
curves} in the shape of \textit{scrolls} lying on
\textit{hyperplanes} the equations of which are well-know. So, it is
possible to analytically compute these \textit{hyperplanes}
equations while using the \textit{curvature of the flow} and then
the comparison leads to a total identity between both equations.
Moreover, it is established in the case of piecewise linear models
that such \textit{hyperplanes} are no more than \textit{osculating
hyperplanes} the invariance of which is stated according to
\textit{Darboux Theorem}. Then, \textit{slow invariant manifolds}
analytical equations of nonlinear high-dimensional dynamical systems
such as fourth-order and fifth-order cubic Chua's circuit [Liu
\textit{et al.}, 2007, Hao \textit{et al.}, 2005]and fifth-order
magnetoconvection system [Knobloch et al., 1981] are directly
provided while using the \textit{curvature of the flow} and
\textit{Darboux Theorem}.

In the discussion, a comparison with various methods of \textit{slow invariant manifold} analytical equation
determination such as \textit{Tangent Linear System Approximation} [Rossetto \textit{et al.}, 1998] and \textit{Geometric Singular Perturbation Theory} [Fenichel, 1979] highlights
that, since it uses neither eigenvectors nor asymptotic expansions but
simply involves time derivatives of the velocity vector field, \textit{curvature of the flow} constitutes
a general method simplifying and improving the \textit{slow invariant manifold} analytical equation
determination of any high-dimensional dynamical systems.

In the appendix, definitions inherent to \textit{Differential
Geometry} such as the concept of $n$-dimensional\textit{ smooth
curves}, \textit{generalized} \textit{Fr\'{e}net frame} and
\textit{curvatures} definitions are briefly recalled as well as the
Gram-Schmidt orthogonalization process for computing
\textit{curvatures} of \textit{trajectory curves} in Euclidean
$n$-space. Then, it is shown that the \textit{flow curvature method}
generalizes the \textit{Tangent Linear System Approximation}
[Rossetto \textit{et al.}, 1998] and encompasses the so-called
\textit{Geometric Singular Perturbation Theory} [Fenichel, 1979].

\section{Slow invariant manifold analytical equation}

The concept of \textit{invariant manifolds} plays a very important role in the stability and structure
of dynamical systems and especially for \textit{slow-fast dynamical systems} or \textit{singularly perturbed systems}. Since the beginning of the
twentieth century it has been subject to a wide range of seminal research.
The classical geometric theory developed originally by Andronov [1937],
Tikhonov [1948] and Levinson [1949] stated that \textit{singularly perturbed systems} possess \textit{invariant manifolds} on which
trajectories evolve slowly and toward which nearby orbits contract
exponentially in time (either forward and backward) in the normal
directions. These manifolds have been called asymptotically stable (or
unstable) \textit{slow manifolds}. Then, Fenichel [1971-1979] theory\footnote{ independently
developed in Hirsch \textit{et al.}, [1977]} for the persistence of normally hyperbolic
invariant manifolds enabled to establish the local invariance of \textit{slow manifolds} that
possess both expanding and contracting directions and which were labeled
\textit{slow invariant manifolds}.

Thus, various methods have been developed in order to determine the \textit{slow invariant}
\textit{manifold} analytical equation associated to \textit{singularly perturbed systems}. The essential works of Wasow [1965],
Cole [1968], O'Malley [1974, 1991] and Fenichel [1971-1979] to name but a
few, gave rise to the so-called \textit{Geometric Singular Perturbation Theory} and the problem for finding the \textit{slow invariant manifold} analytical
equation turned into a regular perturbation problem in which one generally
expected, according to O'Malley [1974 p. 78, 1991 p. 21] the asymptotic
validity of such expansion to breakdown. Another method called: \textit{tangent linear system approximation}, developed
by Rossetto \textit{et al.} [1998], consisted in using the presence of a ``fast''
eigenvalue in the functional jacobian matrix of low-dimensional (2 and 3)
dynamical systems. Within the framework of application of the Tikhonov's
theorem [1952], this method used the fact that in the vicinity of the \textit{slow manifold }the
eigenmode associated with the ``fast'' eigenvalue was evanescent. Thus, the
\textit{tangent linear system approximation} method provided the \textit{slow manifold} analytical equation of low-dimensional dynamical
systems according to the ``slow'' eigenvectors of the \textit{tangent linear system}, i.e., according to
the ``slow'' eigenvalues. Nevertheless, the presence of these eigenvalues
(real or complex conjugated) prevented from expressing this equation
explicitly. Also to solve this problem it was necessary to make such
equation independent of the ``slow'' eigenvalues. This could be carried out
by multiplying it by ``conjugated'' equations leading to a \textit{slow manifold} analytical
equation independent of the ``slow'' eigenvalues of the \textit{tangent linear system}. Then, it was
established in [Ginoux \textit{et al.}, 2006] that the resulting equation was identically
corresponding in dimension two to the \textit{curvature} (\textit{first curvature}) of the flow and in dimension three
to the \textit{torsion} (\textit{second curvature}).

So, in this work the new approach proposed by Ginoux \textit{et al.}
[2006] is generalized to high-dimensional dynamical systems. Thus,
the main result of this work established in the next section is that
\textit{curvature of the flow}, i.e., \textit{curvature} of
\textit{trajectory curves} of any $n$-dimensional dynamical system
directly provides its \textit{slow manifold} analytical equation the
\textit{invariance} of which is established according to
\textit{Darboux Theorem}. Since it uses neither eigenvectors nor
asymptotic expansions but simply involves time derivatives of the
velocity vector field, it constitutes a general method simplifying
and improving the \textit{slow invariant manifold }analytical
equation determination of high-dimensional dynamical systems.

\subsection{Slow manifold of high-dimensional dynamical systems}

In the framework of \textit{Differential Geometry}\footnote{ See appendix for defintions}, \textit{trajectory curves} $\vec
{X}\left( t \right)$ integral of $n$-dimensional dynamical systems (\ref{eq1})
satisfying the assumptions of the Cauchy-Lipschitz theorem may be regarded
as $n$-dimensional \textit{smooth curves}, i.e., \textit{smooth curves} in Euclidean $n-$space \textit{parametrized in terms of time}.

\begin{prop}

\textit{The location of the points where the curvature of the flow,
i.e., the curvature of the trajectory curves of any n-dimensional
dynamical system vanishes directly provides its }$\left( {n - 1}
\right)$\textit{-dimensional slow invariant manifold analytical
equation which reads:}

\begin{equation}
\label{eq2} \phi \left( \vec {X} \right) = \dot {\vec {X}} \cdot
\left( {\ddot {\vec {X}} \wedge \dddot {\vec {X}} \wedge \ldots
\wedge \mathop {\vec {X}}\limits^{\left( n \right)} } \right) =
det\left( {\dot {\vec {X}},\ddot {\vec {X}},\dddot {\vec {X}},\ldots
,\mathop {\vec {X}}\limits^{\left( n \right)} } \right) = 0
\end{equation}

\textit{where} $\mathop {\vec {X}}\limits^{\left( n \right)}
$\textit{ represents the time derivatives of }${\vec {X}} = \left[
{x_1 ,x_2 ,...,x_n } \right]^t.$

\end{prop}

\begin{proof}

Let's notice that \textit{inner product }(\ref{eq2}) reads:

\[
\dot {\vec {X}} \cdot \left( {\ddot {\vec {X}} \wedge \dddot {\vec
{X}} \wedge \ldots \wedge \mathop {\vec{X}}\limits^{\left( n
\right)} } \right) = \left[ {\dot {\vec {X}},\ddot {\vec {X}},\ldots
,\mathop {\vec {X}}\limits^{\left( n \right)} } \right]
\]

where $\wedge$ represents the \textit{wedge product}.
Let's consider the vectors $\vec {u}_1 \left( t \right)$, $\vec
{u}_2 \left( t \right)$, \ldots , $\vec {u}_i \left( t \right)$
forming an orthogonal basis defined with the Gram-Schmidt process
[Lichnerowicz, 1950 p. 30, Gluck, 1966]. Then, while using identity
(\ref{eq53}) established in appendix

\begin{equation}
\label{eq3} \left[ {\dot {\vec {X}},\ddot {\vec {X}},\ldots ,\mathop
{\vec{X}}\limits^{\left( n \right)} } \right] = \left\| {\vec {u}_1
} \right\|\left\| {\vec {u}_2 } \right\|\ldots \left\| {\vec {u}_n }
\right\|
\end{equation}

\textit{curvature} (\ref{eq50}) may be written:

\begin{equation}
\label{eq4} \kappa _i = \frac{\left\| {\vec {u}_{i + 1} \left( t
\right)} \right\|}{\left\| {\vec {u}_1 \left( t \right)}
\right\|\left\| {\vec {u}_i \left( t \right)} \right\|} =
\frac{\left[ {\dot {\vec {X}},\ddot {\vec {X}},\ldots ,\mathop {\vec
{X}}\limits^{\left( {i + 1} \right)} } \right]}{\left\| {\vec {u}_1
} \right\|^2\left\| {\vec {u}_2 } \right\|\ldots \left\| {\vec
{u}_{i - 1} } \right\|\left\| {\vec {u}_i } \right\|^2}
\end{equation}

\vspace{0.2in}

First and second \textit{curvature}s of space \textit{curves}, i.e., \textit{curvature} (\ref{eq51}) and \textit{torsion} (\ref{eq52}) may be, for example,
found again. Thus, for $i = 1$ identity (\ref{eq53}) provides: $\left[ {\dot {\vec
{X}},\ddot {\vec {X}}} \right] = \left\| {\vec {u}_1 } \right\|\left\| {\vec
{u}_2 } \right\|$ and \textit{curvature} $\kappa _1$ reads:

\[
\kappa _1 = \frac{\left\| {\vec {u}_2 } \right\|}{\left\| {\vec {u}_1 }
\right\|\left\| {\vec {u}_1 } \right\|} = \frac{\left[ {\dot {\vec
{X}},\ddot {\vec {X}}} \right]}{\left\| {\vec {u}_1 } \right\|^3} =
\frac{\left\| {\dot {\vec {X}} \wedge \ddot {\vec {X}}} \right\|}{\left\|
\dot {\vec {X}} \right\|^3} = \frac{\left\| {\vec {\gamma } \wedge
\overrightarrow V } \right\|}{\left\| {\overrightarrow V } \right\|^3}
\]

\vspace{0.2in}

For $i = 2$, while using identity (\ref{eq53}): $\left[ {\dot {\vec {X}},\ddot
{\vec {X}},\dddot {\vec {X}}} \right] = \left\| {\vec {u}_1 }
\right\|\left\| {\vec {u}_2 } \right\|\left\| {\vec {u}_3 } \right\|$, the
Gram-Schmidt orthogonalization process (\ref{eq49}) for the expression of vectors
$\vec {u}_1 \left( t \right)$ and $\vec {u}_2 \left( t \right)$ and the
Lagrange identity $\left\| {\vec {u}_1 } \right\|^2\left\| {\vec {u}_2 }
\right\|^2 = \left\| {\dot {\vec {X}} \wedge \ddot {\vec {X}}} \right\|^2$
\textit{torsion} $\kappa _2 $ reads:

\[
\kappa _2 = \frac{\left\| {\vec {u}_3 \left( t \right)} \right\|}{\left\|
{\vec {u}_1 \left( t \right)} \right\|\left\| {\vec {u}_2 \left( t \right)}
\right\|} = \frac{\left[ {\dot {\vec {X}},\ddot {\vec {X}},\dddot {\vec
{X}}} \right]}{\left\| {\vec {u}_1 } \right\|^2\left\| {\vec {u}_2 }
\right\|^2} = \frac{\dot {\vec {X}} \cdot \left( {\ddot {\vec {X}} \wedge
\dddot {\vec {X}}} \right)}{\left\| {\dot {\vec {X}} \wedge \ddot {\vec
{X}}} \right\|^2} = - \frac{\dot {\vec {\gamma }} \cdot \left( {\vec {\gamma
} \wedge \overrightarrow V } \right)}{\left\| {\vec {\gamma } \wedge
\overrightarrow V } \right\|^2}
\]

Thus, the location of the point where the \textit{curvature} \textit{of the flow} (\ref{eq4}) vanishes, i.e., the location
of the point where the \textit{inner product} vanishes defines a $\left( {n - 1}
\right)$\textit{-dimensional} manifold associated to any $n$-dimensional dynamical system (\ref{eq1}):

\begin{eqnarray}
\phi \left( \vec {X} \right) & = & \left[ {\vec {X},\dot {\vec
{X}},\ddot {\vec{X}},\ldots ,\mathop {\vec {X}}\limits^{\left( n
\right)} } \right]\nonumber \\
& = & \dot {\vec {X}} \cdot \left( {\ddot {\vec {X}} \wedge \dddot
{\vec {X}} \wedge \ldots \wedge \mathop {\vec
{X}}\limits^{\left( n \right)} } \right)\nonumber \\
& = & det\left( {\dot {\vec {X}},\ddot {\vec {X}},\dddot {\vec
{X}},\ldots ,\mathop {\vec {X}}\limits^{\left( n \right)} } \right)
= 0 \end{eqnarray}

\end{proof}

The invariance of such manifold is then established while using the \textit{Darboux Theorem}
presented below.

\subsection{Darboux invariance theorem}

According to Schlomiuk [1993] and Llibre \textit{et al.} [2007] it seems that in his memoir
entitled: \textit{Sur les \'{e}quations diff\'{e}rentielles alg\'{e}briques du premier ordre et du premier degr\'{e},} Gaston Darboux [1878, p. 71, 1878$_{c}$, p.1012] has been the
first to define the concept of \textit{invariant manifold}. Let's consider a $n$-dimensional dynamical
system (\ref{eq1}) describing ``the motion of a variable point in a space of
dimension $n$.'' Let $\vec {X} = \left[ {x_1 ,x_2 ,\ldots ,x_n } \right]^t$ be
the coordinates of this point and $\overrightarrow V = \left[ {\dot {x}_1
,\dot {x}_2 ,\ldots ,\dot {x}_n } \right]^t$ the corresponding velocity
vector.

\begin{prop}

The \textit{manifold} defined by $\phi \left( \vec {X} \right) = 0$
where $\phi $ is a $C^1$ in an open set U is \textit{invariant} with
respect to the flow of (\ref{eq1}) if there exists a $C^1$ function
denoted $K\left( \vec {X} \right)$ and called cofactor which
satisfies:

\begin{equation}
\label{eq5}
L_{\overrightarrow V } \phi \left( \vec {X} \right) = K\left( \vec {X}
\right)\phi \left( \vec {X} \right)
\end{equation}

for all $\vec {X} \in U$ and with the Lie derivative operator defined as:

\[
L_{\overrightarrow V } \phi = \overrightarrow V \cdot
\overrightarrow \nabla \phi = \sum\limits_{i = 1}^n {\frac{\partial
\phi }{\partial x_i }\dot {x}_i } = \frac{d\phi }{dt}.
\]

In the following \textit{invariance} of the \textit{slow manifold} will be established according to what will be
referred as \textit{Darboux Theorem}.

\end{prop}

\begin{proof}

Lie derivative of the \textit{inner product }(\ref{eq2}) reads:

\begin{equation}
\label{eq6} L_{\overrightarrow V } \phi \left( \vec {X} \right) =
\dot {\vec {X}} \cdot \left( {\ddot {\vec {X}} \wedge \dddot {\vec
{X}} \wedge \ldots \wedge \mathop {\vec {X}}\limits^{\left( {n + 1}
\right)} } \right) = \left[ {\dot {\vec {X}},\ddot {\vec {X}},\dddot
{\vec {X}},\ldots ,\mathop {\vec {X}}\limits^{\left( {n + 1}
\right)} } \right]
\end{equation}

Moreover, starting from the identity $\ddot {\vec {X}} = J\dot {\vec {X}}$
where $J$ is the functional jacobian matrix associated to any $n$-dimensional
dynamical system (\ref{eq1}) it can be established that:

\[
\mathop {\vec {X}}\limits^{\left( {n + 1} \right)} = J^n\dot
{\vec{X}}\mbox{ \qquad if \qquad} \frac{dJ}{dt} = 0
\]

where $J^n$ represents the $n^{th}$power of $J$.\\
\\
As an example, $\ddot {\vec {X}} = J\dot {\vec {X}} \quad
\Leftrightarrow  \quad \vec {\gamma } = J\vec {V}$. Then, it follows
that

\begin{equation}
\label{eq7} \mathop {\vec {X}}\limits^{\left( {n + 1} \right)} =
JJ^{n - 1}\dot {\vec {X}} = J\mathop {\vec {X}}\limits^{\left( n
\right)}
\end{equation}

Replacing $\mathop {\vec {X}}\limits^{\left( {n + 1} \right)} $ in
expression (\ref{eq6}) by Eq. (\ref{eq7}) we have:

\begin{equation}
\label{eq8} L_{\overrightarrow V } \phi \left( \vec {X} \right) =
\dot {\vec {X}} \cdot \left( {\ddot {\vec {X}} \wedge \dddot {\vec
{X}} \wedge \ldots \wedge J\mathop {\vec {X}}\limits^{\left( n
\right)} } \right) = \left[ {\dot {\vec {X}},\ddot {\vec {X}},\dddot
{\vec {X}},\ldots ,J\mathop {\vec {X}}\limits^{\left( n \right)} }
\right]
\end{equation}

Then, identity (\ref{eq59}) established in appendix leads to:

\[
L_{\overrightarrow V } \phi \left( \vec {X} \right) = Tr\left[ J
\right]\dot {\vec {X}} \cdot \left( {\ddot {\vec {X}} \wedge \dddot
{\vec {X}} \wedge \ldots \wedge \mathop {\vec {X}}\limits^{\left( n
\right)} } \right) = Tr\left[ J \right]\phi \left( \vec {X} \right)
= K\left( \vec {X} \right)\phi \left( \vec {X} \right)
\]

where $K\left( \vec {X} \right) = Tr\left[ J \right]$ represents the trace
of the functional jacobian matrix.

So, according to \textit{Darboux Theorem} invariance of the
\textit{slow manifold} analytical equation of any $n$-dimensional
dynamical system is established provided that the functional
jacobian matrix is stationary (7).

\end{proof}

\textbf{\textit{Note.}} Since the \textit{slow invariant manifold} analytical equation (\ref{eq2}) is defined
starting from the velocity vector field all fixed points are belonging to
it.

\section{Chua's piecewise linear models}

It has been established that Chua's piecewise linear models exhibit
\textit{trajectory curves} in the shape of \textit{double scrolls}
lying on \textit{hyperplanes} the equations of which have been
already analytically computed [Chua \textit{et al.}, 1986, Rossetto
1993, Liu \textit{et al.}, 2007]. The aim of this section is first
to provide these \textit{hyperplanes} equations with a classical
method and then with the new one proposed, i.e., with
\textit{curvature of the flow}. A comparison of \textit{hyperplanes}
equations given by both methods leads to a total identity. Then, it
is stated according to \textit{Darboux Theorem} [1878] that these
\textit{hyperplanes} are overflowing invariant with respect to the
flow of Chua's models and are, consequently, \textit{invariant
manifolds}. Moreover, it is also established, in the framework of
the \textit{Differential Geometry}, that such \textit{hyperplanes}
are no more than ``\textit{osculating hyperplanes''}.

\subsection{Three-dimensional Chua's system}

The piecewise linear Chua's circuit [Chua \textit{et al.}, 1986] is
an electronic circuit comprising an inductance $L_1 $, an active
resistor $R$, two capacitors $C_1 $ and $C_2 $, and a nonlinear
resistor. Chua's circuit can be accurately modeled by means of a
system of three coupled first-order ordinary differential equations
in the variables $x_1 \left( t \right)$, $x_2 \left( t \right)$ and
$x_3 \left( t \right)$, which give the voltages in the capacitors
$C_1 $ and $C_2 $, and the intensity of the electrical current in
the inductance $L_1 $, respectively. These equations called
\textit{global unfolding} of Chua's circuit are written in a
dimensionless form:

\begin{equation}
\label{eq9} \overrightarrow V \left( {{\begin{array}{*{20}c}
 {\frac{dx_1 }{dt}} \hfill \\
 {\frac{dx_2 }{dt}} \hfill \\
 {\frac{dx_3 }{dt}} \hfill \\
\end{array} }} \right) = \overrightarrow \Im \left( {{\begin{array}{*{20}c}
 {f_1 \left( {x_1 ,x_3 ,x_3 } \right)} \hfill \\
 {f_2 \left( {x_1 ,x_3 ,x_3 } \right)} \hfill \\
 {f_3 \left( {x_1 ,x_3 ,x_3 } \right)} \hfill \\
\end{array} }} \right) = \left( {{\begin{array}{*{20}c}
 {\alpha \left( {x_2 - x_1 - k\left( {x_1 } \right)} \right)} \hfill \\
 {x_1 - x_2 + x_3 } \hfill \\
 { - \beta x_2 } \hfill \\
\end{array} }} \right)
\end{equation}

The function $k\left( {x_1 } \right)$ describes the electrical response of
the nonlinear resistor, i.e., its characteristics which is a piecewise
linear function defined by:

\begin{equation}
\label{eq10}
k\left( {x_1 } \right) = \left\{ {{\begin{array}{*{20}c}
 {bx_1 + a - b\mbox{ }x_1 \geqslant 1} \hfill \\
 {ax_1 \mbox{ }\left| {x_1 } \right| \leqslant 1} \hfill \\
 {bx_1 - a + b\mbox{ }x_1 \leqslant - 1} \hfill \\
\end{array} }} \right.
\end{equation}

where the real parameters $\alpha $ and $\beta $ determined by the
particular values of the circuit components are in a standard model
$\alpha = 9$, $\beta = 100/7$, $a = - 8/7$ and $b = - 5/7$ and where
the functions $f_i $ are infinitely differentiable with respect to
all $x_i $, and $t$, i.e., are $C^\infty $ functions in a compact E
included in $\mathbb{R}^3$ and with values in $\mathbb{R}$.

\subsubsection{Tangent linear system approximation}

The piecewise linear Chua's circuit has three fixed points around
which the \textit{double scrolls} wind. Thus, each \textit{scroll}
lies on a \textit{plane} passing through a fixed a point. Its
equation may be calculated while using the \textit{Tangent Linear
System Approximation} [Rossetto, 1993] which consists in using the
\textit{fast }eigenvector associated with the \textit{fast}
eigenvalue of the transposed functional Jacobian matrix in order to
define the normal vector to these \textit{planes}.

The transposed functional jacobian matrix of Chua's system (\ref{eq9}) reads:

\[
{ }^tJ = \left( {{\begin{array}{*{20}c}
 { - \alpha \left( {1 + b} \right)} \hfill & 1 \hfill & 0 \hfill \\
 \alpha \hfill & { - 1} \hfill & { - \beta } \hfill \\
 0 \hfill & 1 \hfill & 0 \hfill \\
\end{array} }} \right)
\]

The \textit{fast }eigenvector associated with the \textit{fast} eigenvalue $\lambda _1 $ may be
written:

\[
{ }^t\overrightarrow {Y_{\lambda _1 } } \left( {{\begin{array}{*{20}c}
 1 \hfill \\
 {\lambda _1 + \alpha \left( {b + 1} \right)} \hfill \\
 {1 + \alpha \frac{b + 1}{\lambda _1 }} \hfill \\
\end{array} }} \right)
\]

Let's denote $\overrightarrow {IM} \left( {x - x_I ,y - y_I ,z - z_I
} \right)$ where $I$ is any fixed point $I_1 $ or $I_2 $ and $M$ any
point belonging to the phase space.\\
It may be checked that:
$\overrightarrow V = J\overrightarrow {IM} $

Thus, according to this method, the $\left( \Pi \right)$ \textit{plane} equation passing
through the fixed point $I_1 $ (resp. $I_2 )$ may be given by the following
orthogonality condition:

\begin{equation}
\label{eq11}
\Pi \left( \vec {X} \right) = \overrightarrow V \cdot { }^t\overrightarrow
{Y_{\lambda _1 } } = 0
\end{equation}

But since $\overrightarrow V = J\overrightarrow {IM} $, Eq. (\ref{eq11}) reads: $\Pi
\left( \vec {X} \right) = J\overrightarrow {IM} \cdot { }^t\overrightarrow
{Y_{\lambda _1 } } = 0$.

Then, according to the \textit{eigenequation}: ${ }^tJ{ }^t\overrightarrow {Y_{\lambda _1 } } =
\lambda _1 { }^t\overrightarrow {Y_{\lambda _1 } } $, it may be checked
that:

\begin{equation}
\label{eq12}
\overrightarrow V \cdot { }^t\overrightarrow {Y_{\lambda _1 } } = \lambda _1
{ }^t\overrightarrow {Y_{\lambda _1 } } \cdot \overrightarrow {IM}
\end{equation}

So, the $\left( \Pi \right)$ \textit{plane} equation passing through the fixed point $I_1
$ (resp. $I_2 )$ is given by:

\begin{equation}
\label{eq13}
\Pi \left( \vec {X} \right) = \lambda _1 \overrightarrow {IM} \cdot {
}^t\overrightarrow {Y_{\lambda _1 } } = 0
\end{equation}

The Lie derivative of $\Pi \left( \vec {X} \right)$ reads, taking into
account Eq. (\ref{eq12}) {\&} Eq. (\ref{eq13}):

\[
L_{\vec {V}} \Pi \left( \vec {X} \right) = \lambda _1
\frac{d\overrightarrow {IM} }{dt} \cdot { }^t\overrightarrow
{Y_{\lambda _1 } } = \lambda _1 \overrightarrow V \cdot {
}^t\overrightarrow {Y_{\lambda _1 } } = \lambda _1 \left( {\lambda
_1 \overrightarrow {IM} \cdot { }^t\overrightarrow {Y_{\lambda _1 }
} } \right) = \lambda _1 \Pi \left( \vec {X} \right)
\]

So, according to \textit{Darboux Theorem }[1878], the \textit{plane} $\Pi \left( \vec {X} \right)$ is invariant.

\subsubsection{Curvature of the flow}

\textit{Curvature of the flow} states that the location of the points where the \textit{second curvature} (\textit{torsion}) of the flow, i.e., the
\textit{second curvature }of the \textit{trajectory curves }integral of Chua's system vanishes directly provides its \textit{slow} \textit{invariant} \textit{manifold}
analytical equation, i.e., the $\left( \Pi \right)$ \textit{planes} equations. According to
Proposition 3.1, Eq. (\ref{eq2}) may be written:

\begin{equation}
\label{eq14}
\phi \left( \vec {X} \right) = \overrightarrow V \cdot \left( {\vec {\gamma
} \wedge \dot {\vec {\gamma }}} \right) = 0
\end{equation}

It can been easily established for any dynamical system that: $\vec {\gamma
} = J\overrightarrow V $. Moreover, since the Chua's system (\ref{eq9}) is
piecewise linear the time derivative of the functional jacobian matrix is
zero: $\frac{dJ}{dt} = 0$. As a consequence, the over-acceleration (or jerk)
reads: $\dot {\vec {\gamma }} = J\vec {\gamma } +
\frac{dJ}{dt}\overrightarrow V = J\vec {\gamma }$.

But since $\overrightarrow V = J\overrightarrow {IM} $, the \textit{slow manifold} equation (\ref{eq14})
may be written:

\begin{equation}
\label{eq15}
\phi \left( \vec {X} \right) = J\overrightarrow {IM} \cdot \left(
{J\overrightarrow V \wedge J\vec {\gamma }} \right) = 0
\end{equation}

The identity (\ref{eq58}) $J\vec {a}.\left( {J\vec {b} \wedge J\vec {c}} \right) =
Det\left( J \right)\vec {a}.\left( {\vec {b} \wedge \vec {c}} \right)$
established in appendix leads to:

\begin{equation}
\label{eq16}
\phi \left( \vec {X} \right) = Det\left( J \right)\overrightarrow {IM} \cdot
\left( {\overrightarrow V \wedge \vec {\gamma }} \right) = 0
\end{equation}

where, $\overrightarrow {IM} \cdot \left( {\overrightarrow V \wedge \vec
{\gamma }} \right) = 0$ is the \textit{osculating plane} passing through the fixed point $I_1 $
(resp. $I_2 )$.

\newpage

The Lie derivative of $\phi \left( \vec {X} \right)$ reads, taking into
account that $\dot {\vec {\gamma }} = J\vec {\gamma }$

\begin{equation}
\label{eq17} L_{\vec {V}} \phi \left( \vec {X} \right) = Det\left( J
\right)\overrightarrow {IM} \cdot \left( {\overrightarrow V \wedge
\dot {\vec {\gamma }}} \right) = Det\left( J \right)\overrightarrow
{IM} \cdot \left( {\overrightarrow V \wedge J\vec {\gamma }} \right)
= 0
\end{equation}

The identity (\ref{eq59}) $J\vec {a}.\left( {\vec {b} \wedge \vec
{c}} \right) + \vec {a}.\left( {J\vec {b} \wedge \vec {c}} \right) +
\vec {a}.\left( {\vec {b} \wedge J\vec {c}} \right) = Tr\left( J
\right)\vec {a}.\left( {\vec {b} \wedge \vec {c}} \right)$
established in appendix leads to:

\[
\overrightarrow {IM} \cdot \left( {\overrightarrow V \wedge J\vec
{\gamma }\,} \right) = Tr\left( J \right)\overrightarrow {IM} \cdot
\left( {\overrightarrow V \wedge \vec {\gamma }} \right) \mbox{\quad
and \quad} L_{\vec {V}} \phi \left( \vec {X} \right) = Tr\left[ J
\right]\phi \left( \vec {X} \right).
\]

So, according to \textit{Darboux Theorem }[1878], the manifold $\phi \left( \vec {X} \right)$ is
invariant.

Moreover, while multiplying Eq. (\ref{eq11}) by its ``conjugated'' equations, i.e.,
by $\overrightarrow V \cdot { }^t\overrightarrow {Y_{\lambda _2 } } $ and
$\overrightarrow V \cdot { }^t\overrightarrow {Y_{\lambda _3 } } $ we have:

\begin{equation}
\label{eq18}
\left( {\overrightarrow V \cdot { }^t\overrightarrow {Y_{\lambda _1 } } }
\right)\left( {\overrightarrow V \cdot { }^t\overrightarrow {Y_{\lambda _2 }
} } \right)\left( {\overrightarrow V \cdot { }^t\overrightarrow {Y_{\lambda
_3 } } } \right) = 0
\end{equation}

But, it has been established [Ginoux \textit{et al.}, 2006] that Eq. (\ref{eq18}) is totally
identical to Eq. (\ref{eq14}). So, taking into account Eq. (\ref{eq12}) {\&} (\ref{eq13}), it may be
written:

\[
\Pi \left( \vec {X} \right)\left( {\overrightarrow V \cdot {
}^t\overrightarrow {Y_{\lambda _2 } } } \right)\left( {\overrightarrow V
\cdot { }^t\overrightarrow {Y_{\lambda _3 } } } \right) = \overrightarrow V
\cdot \left( {\vec {\gamma } \wedge \dot {\vec {\gamma }}} \right) = 0
\]

Then, it proves that the $\left( \Pi \right)$ \textit{plane} equation (\ref{eq13}) is in factor in
Eq. (\ref{eq14}) and so that both methods provide the same\textit{ planes} equations. Moreover, in
the framework of \textit{Differential Geometry}, the $\left( \Pi \right)$ \textit{plane} may be interpreted as the
\textit{osculating plane} passing through each fixed point $I_1 $ (resp. $I_2 )$.

\begin{figure}[htbp]
\centerline{\includegraphics{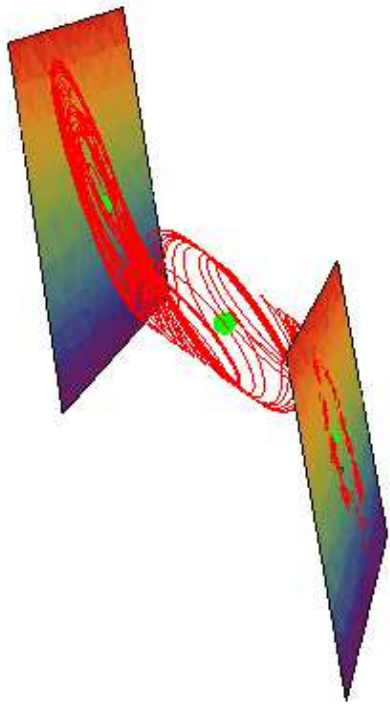}}
\caption{Chua's chaotic invariant hyperplanes for: $\alpha = 9$, $\beta = 100/7$, $a =  -
8/7$, $b = - 5/7.$}
\label{fig1}
\end{figure}

With this set of parameters: $\lambda _1 = - 3.9421$ ; ${
}^t\overrightarrow {Y_{\lambda _1 } } \left( {2.8759, -
3.9421,\mbox{1}} \right)$

$\left( {\Pi _{1,2} } \right)$ \textit{hyperplanes} equations passing through the fixed point
$I_{1,2} \left( { \mp 3 \mathord{\left/ {\vphantom {3 2}} \right.
\kern-\nulldelimiterspace} 2,0,{\pm 3} \mathord{\left/ {\vphantom {{\pm 3}
2}} \right. \kern-\nulldelimiterspace} 2} \right)$ given by both methods
read:

\[
\Pi _{1,2} \left( \vec {X} \right) = 2.8759x_1 - 3.9421x_2 + x_3 \pm 2.8139
= 0
\]

and are plotted in Fig. 1.

\newpage

\subsection{Four-dimensional Chua's system}

The piecewise linear fourth-order Chua's circuit [Thamilmaran
\textit{et al.}, 2004] is an electronic circuit comprising two
inductances $L_1 $ and $L_2 $, two linear resistors $R$ and $R_1 $,
two capacitors $C_1 $ and $C_2 $, and a nonlinear resistor.
Fourth-order Chua's circuit can be accurately modeled by means of a
system of four coupled first-order ordinary differential equations
in the variables $x_1 \left( t \right)$, $x_2 \left( t \right)$,
$x_3 \left( t \right)$ and $x_4 \left( t \right)$, which give the
voltages in the capacitors $C_1 $ and $C_2 $, and the intensities of
the electrical current in the inductance $L_1 $ and $L_2 $,
respectively. These equations called \textit{global unfolding} of
Chua's circuit are written in a dimensionless form:

\begin{equation}
\label{eq19} \overrightarrow V \left( {{\begin{array}{*{20}c}
 {\frac{dx_1 }{dt}} \hfill \\
 {\frac{dx_2 }{dt}} \hfill \\
 {\frac{dx_3 }{dt}} \hfill \\
 {\frac{dx_4 }{dt}} \hfill \\
\end{array} }} \right) = \overrightarrow \Im \left( {{\begin{array}{*{20}c}
 {f_1 \left( {x_1 ,x_3 ,x_3 ,x_4 } \right)} \hfill \\
 {f_2 \left( {x_1 ,x_3 ,x_3 ,x_4 } \right)} \hfill \\
 {f_3 \left( {x_1 ,x_3 ,x_3 ,x_4 } \right)} \hfill \\
 {f_4 \left( {x_1 ,x_3 ,x_3 ,x_4 } \right)} \hfill \\
\end{array} }} \right) = \left( {{\begin{array}{*{20}c}
 {\alpha _1 \left( {x_3 - k\left( {x_1 } \right)} \right)} \hfill \\
 {\alpha _2 x_2 - x_3 - x_4 } \hfill \\
 {\beta _1 \left( {x_2 - x_1 - x_3 } \right)} \hfill \\
 {\beta _2 x_2 } \hfill \\
\end{array} }} \right)
\end{equation}

The function $k\left( {x_1 } \right)$ describes the electrical response of
the nonlinear resistor, i.e., its characteristics which is a piecewise
linear function defined by:

\begin{equation}
\label{eq20}
k\left( {x_1 } \right) = \left\{ {{\begin{array}{*{20}c}
 {bx_1 + a - b\mbox{ }x_1 \geqslant 1} \hfill \\
 {ax_1 \mbox{ }\left| {x_1 } \right| \leqslant 1} \hfill \\
 {bx_1 - a + b\mbox{ }x_1 \leqslant - 1} \hfill \\
\end{array} }} \right.
\end{equation}

where the real parameters $\alpha _i $ and $\beta _i $ determined by
the particular values of the circuit components are in a standard
model $\alpha _1 = 2.1429$, $\alpha _2 = - 0.18$, $\beta _1 =
0.0774$, $\beta _2 = 0.003$ $a = - 0.42$, $b = 1.2$ and where the
functions $f_i $ are infinitely differentiable with respect to all
$x_i $, and $t$, i.e., are $C^\infty $ functions in a compact E
included in $\mathbb{R}^4$ and with values in $\mathbb{R}$.

\subsubsection{Tangent linear system approximation}

The fourth-order piecewise linear Chua's circuit has three fixed points
around which the \textit{double scroll} winds in a hyperspace of dimension four. In a reduced
phase space of dimension three, each \textit{scroll} lies on a \textit{hyperplane} passing through a fixed a
point the equation of which may be calculated while using the \textit{Generalized} \textit{Tangent Linear System Approximation} presented in
appendix. So, according to this method, the $\left( \Pi \right)$
\textit{hyperplane} equation passing through the fixed point $I_1 $ (resp. $I_2 )$ is given by
the following orthogonality condition:

\begin{equation}
\label{eq21}
\Pi \left( \vec {X} \right) = \overrightarrow V \cdot { }^t\overrightarrow
{Y_{\lambda _1 } } = 0
\end{equation}

The piecewise linear feature enables to extend the results of the previous
Sec. 3.1. to higher dimensions. So, the $\left( \Pi \right)$ \textit{hyperplanes} equations
passing through the fixed point $I_1 $ (resp. $I_2 )$ is given by:

\begin{equation}
\label{eq22}
\Pi \left( \vec {X} \right) = \lambda _1 \overrightarrow {IM} \cdot {
}^t\overrightarrow {Y_{\lambda _1 } } = 0
\end{equation}

The Lie derivative of $\Pi \left( \vec {X} \right)$ reads:

\[
L_{\vec {V}} \Pi \left( \vec {X} \right) = \lambda _1
\frac{d\overrightarrow {IM} }{dt} \cdot { }^t\overrightarrow
{Y_{\lambda _1 } } = \lambda _1 \overrightarrow V \cdot {
}^t\overrightarrow {Y_{\lambda _1 } } = \lambda _1 \left( {\lambda
_1 \overrightarrow {IM} \cdot { }^t\overrightarrow {Y_{\lambda _1 }
} } \right) = \lambda _1 \Pi \left( \vec {X} \right)
\]

So, according to \textit{Darboux Theorem }[1878], the \textit{hyperplane} $\Pi \left( \vec {X} \right)$ is invariant.

\subsubsection{Curvature of the flow}

\textit{Curvature of the flow} states that the location of the points where the \textit{third curvature} of the flow, i.e., the
\textit{third curvature }of the \textit{trajectory curves }integral of Chua's fourth-order system vanishes directly provides its
\textit{slow} \textit{invariant} \textit{manifold} analytical equation, i.e., the $\left( \Pi \right)$ \textit{hyperplanes} equations. According
to Proposition 3.1, Eq. (\ref{eq2}) may be written:

\begin{equation}
\label{eq23}
\phi \left( \vec {X} \right) = \overrightarrow V \cdot \left( {\vec {\gamma
} \wedge \dot {\vec {\gamma }} \wedge \ddot {\vec {\gamma }}} \right) = 0
\end{equation}

The piecewise linear feature enables to state that: $\mathop {\vec
{\gamma}}\limits^{\left( n \right)} = J^{\left( {n + 1}
\right)}\overrightarrow V = J^{\left( n \right)}\vec {\gamma }$. So,
according to the fact that as previously: $\overrightarrow V =
J\overrightarrow {IM} $, the \textit{slow manifold} equation
(\ref{eq23}) reads:

\begin{equation}
\label{eq24}
\phi \left( \vec {X} \right) = J\overrightarrow {IM} \cdot \left(
{J\overrightarrow V \wedge J\vec {\gamma } \wedge J\dot {\vec {\gamma }}}
\right) = 0
\end{equation}

Identity (\ref{eq58}) $J\vec {a}.\left( {J\vec {b} \wedge J\vec {c} \wedge J\vec
{d}} \right) = Det\left( J \right)\vec {a}.\left( {\vec {b} \wedge \vec {c}
\wedge \vec {d}} \right)$ established in appendix leads to:

\begin{equation}
\label{eq25}
\phi \left( \vec {X} \right) = Det\left( J \right)\overrightarrow {IM} \cdot
\left( {\overrightarrow V \wedge \vec {\gamma } \wedge \dot {\vec {\gamma
}}} \right) = 0
\end{equation}

where, $\overrightarrow {IM} \cdot \left( {\overrightarrow V \wedge \vec
{\gamma } \wedge \dot {\vec {\gamma }}} \right) = 0$ is the \textit{osculating plane} passing through
the fixed point $I_1 $ (resp. $I_2 )$. The Lie derivative of $\phi \left(
\vec {X} \right)$ reads, taking into account that $\dot {\vec {\gamma }} =
J\vec {\gamma }$ and $\ddot {\vec {\gamma }} = J\dot {\vec {\gamma }}$

\begin{equation}
\label{eq26} L_{\vec {V}} \phi \left( \vec {X} \right) = Det\left( J
\right)\overrightarrow {IM} \cdot \left( {\overrightarrow V \wedge
\vec {\gamma } \wedge \ddot {\vec {\gamma }}} \right) = Det\left( J
\right)\overrightarrow {IM} \cdot \left( {\overrightarrow V \wedge
\vec {\gamma } \wedge J\dot {\vec {\gamma }}} \right) = 0
\end{equation}

The identity (\ref{eq59}) established in appendix:

\[
J\vec {a}.\left( {\vec {b} \wedge \vec {c} \wedge \vec {d}} \right) + \vec
{a}.\left( {J\vec {b} \wedge \vec {c} \wedge \vec {d}} \right) + \vec
{a}.\left( {\vec {b} \wedge J\vec {c} \wedge \vec {d}} \right) + \vec
{a}.\left( {\vec {b} \wedge \vec {c} \wedge J\vec {d}} \right) = Tr\left( J
\right)\vec {a}.\left( {\vec {b} \wedge \vec {c} \wedge \vec {d}} \right)
\]

leads to: $\overrightarrow {IM} \cdot \left( {\overrightarrow V
\wedge \vec {\gamma }\, \wedge J\dot {\vec {\gamma }}} \right) =
Tr\left( J \right)\overrightarrow {IM} \cdot \left( {\overrightarrow
V \wedge {\vec {\gamma }} \wedge \dot {\vec {\gamma }}} \right)$ and
$L_{\vec {V}} \phi \left( \vec {X} \right) = Tr\left[ J \right]\phi
\left( \vec {X} \right)$.

Moreover, while multiplying $\overrightarrow V \cdot {
}^t\overrightarrow {Y_{\lambda _1 } }$ by its ``conjugated''
equations, i.e., by $\overrightarrow V \cdot { }^t\overrightarrow
{Y_{\lambda _2 } } $, $\overrightarrow V \cdot { }^t\overrightarrow
{Y_{\lambda _3 } } $ and $\overrightarrow V \cdot {
}^t\overrightarrow {Y_{\lambda _4 } } $ we have:

\begin{equation}
\label{eq27}
\left( {\overrightarrow V \cdot { }^t\overrightarrow {Y_{\lambda _1 } } }
\right)\left( {\overrightarrow V \cdot { }^t\overrightarrow {Y_{\lambda _2 }
} } \right)\left( {\overrightarrow V \cdot { }^t\overrightarrow {Y_{\lambda
_3 } } } \right)\left( {\overrightarrow V \cdot { }^t\overrightarrow
{Y_{\lambda _4 } } } \right) = 0
\end{equation}

It may also be established that Eq. (\ref{eq27}) is totally identical to Eq. (\ref{eq23})
and so that

\[
\Pi \left( \vec {X} \right)\left( {\overrightarrow V \cdot {
}^t\overrightarrow {Y_{\lambda _2 } } } \right)\left( {\overrightarrow V
\cdot { }^t\overrightarrow {Y_{\lambda _3 } } } \right)\left(
{\overrightarrow V \cdot { }^t\overrightarrow {Y_{\lambda _4 } } } \right) =
\overrightarrow V \cdot \left( {\vec {\gamma } \wedge \dot {\vec {\gamma }}
\wedge \ddot {\vec {\gamma }}} \right) = 0
\]

Then, it proves that the $\left( \Pi \right)$ \textit{hyperplane} equation (\ref{eq22}) is in factor in
Eq. (\ref{eq23}) and so that both methods provide the same\textit{ hyperplanes} equations. Moreover, in
the framework of \textit{Differential Geometry}, the $\left( \Pi \right)$ \textit{hyperplane} may be interpreted as the
\textit{osculating hyperplane} passing through each fixed point $I_1 $ (resp. $I_2 )$.

\begin{figure}[htbp]
\centerline{\includegraphics{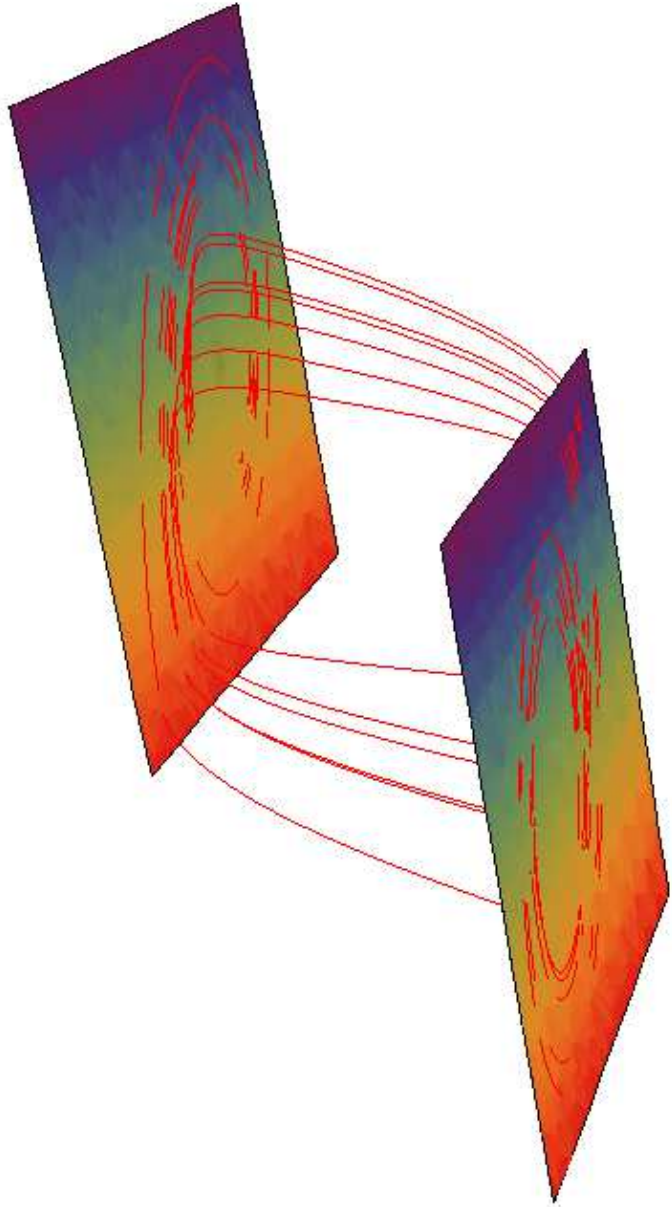}} \caption{Chua's fourth-order
\textit{invariant hyperplanes} in ($x_1x_2x_3$) space for: $\alpha
_1 = 2.1429, \alpha _2 = - 0.18, \beta _1 = 0.0774, \beta _2 =
0.003, a = - 0.42, b = 1.2.$}
\label{fig2}
\end{figure}

With this set of parameters:

\begin{itemize}

\item $\lambda _1 = - 2.5039$ ;

\item ${ }^t\overrightarrow {Y_{\lambda _1 } } \left( { - 0.7532, -
0.01895,0.6574, - 0.007568} \right)$

\end{itemize}

$\left( {\Pi _{1,2} } \right)$ \textit{hyperplanes} equations
passing through fixed point $I_{1,2} \left( { \mp 0.7363,0,\pm
0.7363, \mp 0.7363} \right)$ given by both methods read:

\[
\Pi _{1,2} \left( \vec {X} \right) = 1.8861x_1 + 0.04744x_2 - 1.6461x_3 +
0.01895x_4 \pm 2.6149 = 0
\]

and are plotted in Fig. 2.

\newpage

\subsection{Five-dimensional Chua's system}

The piecewise linear fifth-order Chua's circuit [Hao \textit{et
al.}, 2005] is built while adding a RLC parallel circuit into the
L-arm of Chua's circuit. This electronic circuit consists of two
inductances $L_1 $ and $L_2 $, two linear resistors $R$ and $R_1 $,
three capacitors $C_1 $, $C_2 $ and $C_3 $, and a nonlinear
resistor. Fifth-order Chua's circuit can be accurately modeled by
means of a system of five coupled first-order ordinary differential
equations in the variables $x_1 \left( t \right)$, $x_2 \left( t
\right)$, $x_3 \left( t \right)$, $x_4 \left( t \right)$ and $x_5
\left( t \right)$, which give the voltages in the capacitors $C_1 $,
$C_2 $ and $C_3 $, and the intensities of the electrical current in
the inductance $L_1 $ and $L_2 $, respectively. These equations
called \textit{global unfolding} of Chua's circuit are written in a
dimensionless form:

\begin{equation}
\label{eq28} \overrightarrow V \left( {{\begin{array}{*{20}c}
 {\frac{dx_1 }{dt}} \hfill \\
 {\frac{dx_2 }{dt}} \hfill \\
 {\frac{dx_3 }{dt}} \hfill \\
 {\begin{array}{l}
 \frac{dx_4 }{dt} \\
 \frac{dx_5 }{dt} \\
 \end{array}} \hfill \\
\end{array} }} \right) = \overrightarrow \Im \left( {{\begin{array}{*{20}c}
 {f_1 \left( {x_1 ,x_3 ,x_3 ,x_4 ,x_5 } \right)} \hfill \\
 {f_2 \left( {x_1 ,x_3 ,x_3 ,x_4 ,x_5 } \right)} \hfill \\
 {f_3 \left( {x_1 ,x_3 ,x_3 ,x_4 ,x_5 } \right)} \hfill \\
 {\begin{array}{l}
 f_4 \left( {x_1 ,x_3 ,x_3 ,x_4 ,x_5 } \right) \\
 f_5 \left( {x_1 ,x_3 ,x_3 ,x_4 ,x_5 } \right) \\
 \end{array}} \hfill \\
\end{array} }} \right) = \left( {{\begin{array}{*{20}c}
 {\alpha _1 \left( {x_2 - x_1 - k\left( {x_1 } \right)} \right)} \hfill \\
 {\alpha _2 x_1 - x_2 + x_3 } \hfill \\
 {\beta _1 \left( {x_4 - x_2 } \right)} \hfill \\
 {\begin{array}{l}
 \beta _2 \left( {x_3 + x_5 } \right) \\
 \gamma _2 \left( {x_4 + \gamma _1 x_5 } \right) \\
 \end{array}} \hfill \\
\end{array} }} \right)
\end{equation}

The function $k\left( {x_1 } \right)$ describes the electrical response of
the nonlinear resistor, i.e., its characteristics which is a piecewise
linear function defined by:

\begin{equation}
\label{eq29}
k\left( {x_1 } \right) = \left\{ {{\begin{array}{*{20}c}
 {bx_1 + a - b\mbox{ }x_1 \geqslant 1} \hfill \\
 {ax_1 \mbox{ }\left| {x_1 } \right| \leqslant 1} \hfill \\
 {bx_1 - a + b\mbox{ }x_1 \leqslant - 1} \hfill \\
\end{array} }} \right.
\end{equation}

where the real parameters $\alpha _i $, $\beta _i $ and $\gamma _i $
determined by the particular values of the circuit components are: $\alpha
_1 = 9.934$, $\alpha _2 = 1$, $\beta _1 = 14.47$, $\beta _2 = - 406.5$,
$\gamma _1 = - 0.0152$, $\gamma _2 = 41000$, $a = - 1.246$, $b = - 0.6724$
and where the functions $f_i $ are infinitely differentiable with respect to
all $x_i $, and $t$, i.e., are $C^\infty $ functions in a compact E included in
$\mathbb{R}^5$ and with values in $\mathbb{R}$.

\subsubsection{Tangent linear system approximation}

The fifth-order piecewise linear Chua's circuit has three fixed points
around which the \textit{double scroll} winds in a hyper space of dimension five. In a reduced
phase space of dimension three, each \textit{scroll} lies on a $\left( \Pi \right)$
\textit{hyperplane} equation passing through the fixed point $I_1 $ (resp. $I_2 )$ the equation
of which may be calculated still using the \textit{Generalized} \textit{Tangent Linear System Approximation} presented in appendix. So, the
following orthogonality condition leads to:

\begin{equation}
\label{eq30}
\Pi \left( \vec {X} \right) = \overrightarrow V \cdot { }^t\overrightarrow
{Y_{\lambda _1 } } = 0
\end{equation}

The piecewise linear feature still enables to extend the results of the
previous Sec. 3.1. to higher dimensions. So, the $\left( \Pi \right)$
\textit{hyperplanes} passing through the fixed point $I_1 $ (resp. $I_2 )$ are invariant
according to \textit{Darboux Theorem }[1878].

\subsubsection{Curvature of the flow}

\textit{Curvature of the flow} states that the location of the points where the \textit{fourth curvature} of the flow, i.e., the
\textit{fourth curvature }of the \textit{trajectory curves }integral of Chua's fifth-order system vanishes directly provides its
\textit{slow} \textit{invariant} \textit{manifold} analytical equation, i.e., the $\left( \Pi \right)$ \textit{hyperplanes} equations. According
to Proposition 3.1, Eq. (\ref{eq2}) may be written:

\begin{equation}
\label{eq31}
\phi \left( \vec {X} \right) = \overrightarrow V \cdot \left( {\vec {\gamma
} \wedge \dot {\vec {\gamma }} \wedge \ddot {\vec {\gamma }} \wedge \dddot
{\vec {\gamma }}} \right) = 0
\end{equation}

The piecewise linear feature and both identity (\ref{eq58}) and (\ref{eq59}) enable to
state, according to \textit{Darboux Theorem }[1878], that the manifold $\phi \left( \vec {X} \right)$
is invariant.

Morover, while multiplying $\overrightarrow V \cdot { }^t\overrightarrow
{Y_{\lambda _1 } } $ by its ``conjugated'' equations, i.e., by
$\overrightarrow V \cdot { }^t\overrightarrow {Y_{\lambda _2 } } $,
$\overrightarrow V \cdot { }^t\overrightarrow {Y_{\lambda _3 } } $,
$\overrightarrow V \cdot { }^t\overrightarrow {Y_{\lambda _4 } } $ and
$\overrightarrow V \cdot { }^t\overrightarrow {Y_{\lambda _5 } } $ we have:

\begin{equation}
\label{eq32}
\left( {\overrightarrow V \cdot { }^t\overrightarrow {Y_{\lambda _1 } } }
\right)\left( {\overrightarrow V \cdot { }^t\overrightarrow {Y_{\lambda _2 }
} } \right)\left( {\overrightarrow V \cdot { }^t\overrightarrow {Y_{\lambda
_3 } } } \right)\left( {\overrightarrow V \cdot { }^t\overrightarrow
{Y_{\lambda _4 } } } \right)\left( {\overrightarrow V \cdot {
}^t\overrightarrow {Y_{\lambda _5 } } } \right) = 0
\end{equation}

It may also be established that Eq. (\ref{eq32}) is totally identical to Eq. (\ref{eq31})
and so that

\[
\Pi \left( \vec {X} \right)\left( {\overrightarrow V \cdot {
}^t\overrightarrow {Y_{\lambda _2 } } } \right)\left( {\overrightarrow V
\cdot { }^t\overrightarrow {Y_{\lambda _3 } } } \right)\left(
{\overrightarrow V \cdot { }^t\overrightarrow {Y_{\lambda _4 } } }
\right)\left( {\overrightarrow V \cdot { }^t\overrightarrow {Y_{\lambda _5 }
} } \right) = \overrightarrow V \cdot \left( {\vec {\gamma } \wedge \dot
{\vec {\gamma }} \wedge \ddot {\vec {\gamma }} \wedge \dddot {\vec {\gamma
}}} \right) = 0
\]

Then, it proves that the $\left( \Pi \right)$ \textit{hyperplane} equation (\ref{eq30}) is in factor in
Eq. (\ref{eq31}) and so that both methods provide the same\textit{ hyperplanes} equations. Moreover, in
the framework of \textit{Differential Geometry}, the $\left( \Pi \right)$ \textit{hyperplane} may still be interpreted as the
\textit{osculating hyperplane} passing through each fixed point $I_1 $ (resp. $I_2 )$.

\vspace{0.1in}

With this set of parameters eigenvalues and eigenvectors are respectively:

\begin{itemize}

\item $ \lambda _1 = - 311.49$

\item ${ }^t\overrightarrow {Y_{\lambda _1 } } \left( {0.5625, -
0.8068,0.1804,0.00009693, - 0.000063709} \right)$

\end{itemize}

$\left( {\Pi _{1,2} } \right)$ \textit{hyperplanes} equations passing through fixed point

\[
I_{1,2} \left( { \mp 1.83477, \mp 0.027471,\pm 1.8073, \mp 0.027471, \mp
1.8073} \right)
\]

given by both methods read:

\[
\Pi _{1,2} \left( \vec {X} \right) = - 2.63746x_1 + 3.78315x_2 - 0.846258x_3
- 0.000454517x_4 + 0.000298719x_5 \mp 3.20524
\]

and are plotted in Fig. 3.

\begin{figure}[htbp] \centerline{\includegraphics{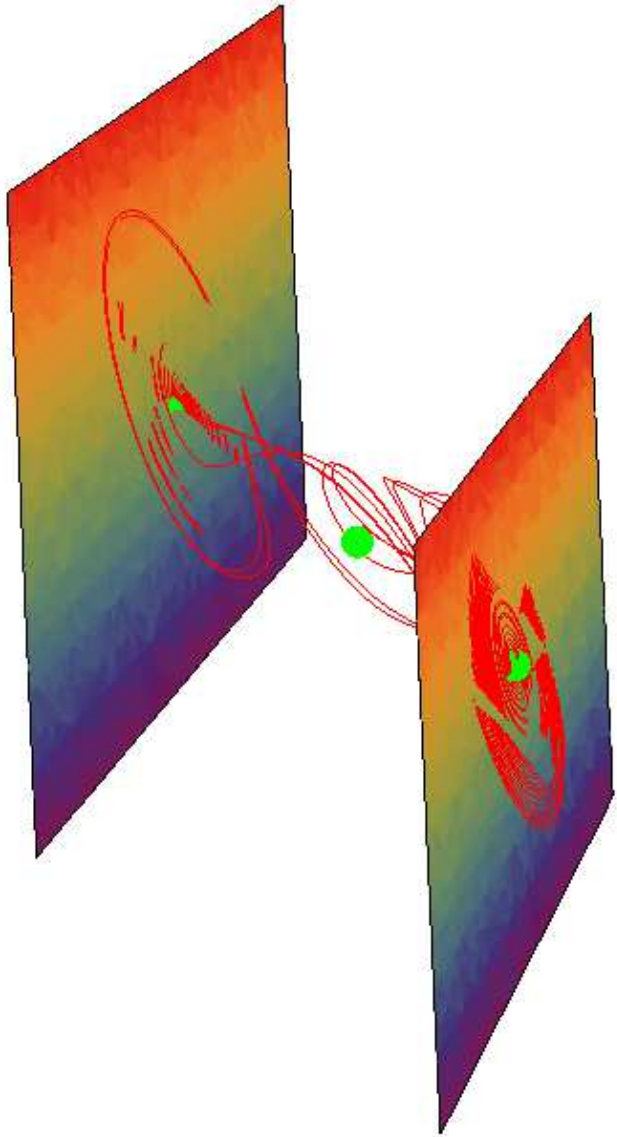}}
\caption{Chua's fifth-order \textit{invariant hyperplanes} in
($x_1x_2x_3$) space for: $\alpha _2 = 1$, $\alpha _1 = 9.934$,
$\beta _1 = 14.47$, $\beta _2 = - 406.5$, $\gamma _1 = - 0.0152$,
$\gamma _2 = 41000$, $a = - 1.246$, $b = - 0.6724.$}
\label{fig3}
\end{figure}

\section{Chua's cubic nonlinear models}

After these tutorial examples concerning Chua's piecewise linear systems,
let's apply the \textit{curvature of the flow} to \textit{nonlinear} Chua's cubic systems of dimension three, four and five.

\subsection{Three-dimensional cubic Chua's system}

The \textit{slow invariant manifold} of the third-order Chua's cubic circuit [Rossetto \textit{et al.}, 1998] has already
been calculated with \textit{curvature of the flow} in [Ginoux \textit{et al.}, 2006].

\subsection{Four-dimensional cubic Chua's system}

The fourth-order cubic Chua's circuit [Thamilmaran \textit{et al.}, 2004, Liu \textit{et al.}, 2007] may
be described starting from the same set of differential equations as (\ref{eq19})
but while replacing the piecewise linear function by a smooth cubic
nonlinear.

\begin{equation}
\label{eq33} \overrightarrow V \left( {{\begin{array}{*{20}c}
 {\frac{dx_1 }{dt}} \hfill \\
 {\frac{dx_2 }{dt}} \hfill \\
 {\frac{dx_3 }{dt}} \hfill \\
 {\frac{dx_4 }{dt}} \hfill \\
\end{array} }} \right) = \overrightarrow \Im \left( {{\begin{array}{*{20}c}
 {f_1 \left( {x_1 ,x_3 ,x_3 ,x_4 } \right)} \hfill \\
 {f_2 \left( {x_1 ,x_3 ,x_3 ,x_4 } \right)} \hfill \\
 {f_3 \left( {x_1 ,x_3 ,x_3 ,x_4 } \right)} \hfill \\
 {f_4 \left( {x_1 ,x_3 ,x_3 ,x_4 } \right)} \hfill \\
\end{array} }} \right) = \left( {{\begin{array}{*{20}c}
 {\alpha _1 \left( {x_3 - \hat{k}\left( {x_1 } \right)} \right)} \hfill \\
 {\alpha _2 x_2 - x_3 - x_4 } \hfill \\
 {\beta _1 \left( {x_2 - x_1 - x_3 } \right)} \hfill \\
 {\beta _2 x_2 } \hfill \\
\end{array} }} \right)
\end{equation}

The function $\hat {k}\left( {x_1 } \right)$ describing the electrical
response of the nonlinear resistor is an odd-symmetric function similar to
the piecewise linear nonlinearity $k\left( {x_1 } \right)$ for which the
parameters $c_1 = 0.3937$ and $c_2 = - 0.7235$ are determined while using
least-square method [Tsuneda, 2005] and which characteristics is defined by:

\begin{equation}
\label{eq34}
\hat {k}\left( {x_1 } \right) = c_1 x_1^3 + c_2 x_1
\end{equation}

The real parameters $\alpha _i $ and $\beta _i $ determined by the
particular values of the circuit components are in a standard model $\alpha
_1 = 2.1429$, $\alpha _2 = - 0.18$, $\beta _1 = 0.0774$, $\beta _2 = 0.003$
$c_1 = 0.3937$ and $c_2 = - 0.7235$ and where the functions $f_i $ are
infinitely differentiable with respect to all $x_i $, and $t$, i.e., are
$C^\infty $ functions in a compact E included in $\mathbb{R}^4$ and with
values in $\mathbb{R}$.

\textit{Curvature of the flow} states that the location of the points where the \textit{fourth curvature of the flow}, i.e., the \textit{fourth curvature }of the
\textit{trajectory curves }integral of Chua's cubic system vanishes directly provides its \textit{slow} \textit{invariant} \textit{manifold} analytical
equation. According to Proposition 3.1, Eq. (\ref{eq2}) may be written:

\begin{equation}
\label{eq35}
\phi \left( \vec {X} \right) = \overrightarrow V \cdot \left( {\vec {\gamma
} \wedge \dot {\vec {\gamma }} \wedge \ddot {\vec {\gamma }} \wedge \dddot
{\vec {\gamma }}} \right) = 0
\end{equation}

Then, it may be proved that in the vicinity of the \textit{singular
approximation} defined by $f_1 \left( \vec {X} \right) = 0$ the
functional jacobian matrix is stationary, i.e., its time derivative
vanishes identically and so, Lie derivative $L_{\vec {V}} \phi
\left( \vec {X} \right) = 0$ vanishes identically. Thus, according
to \textit{Darboux Theorem }[1878], the manifold $\phi \left( \vec
{X} \right)$ which is \textit{locally invariant} is plotted in Fig. 4.

\begin{figure}[htbp]
\centerline{\includegraphics{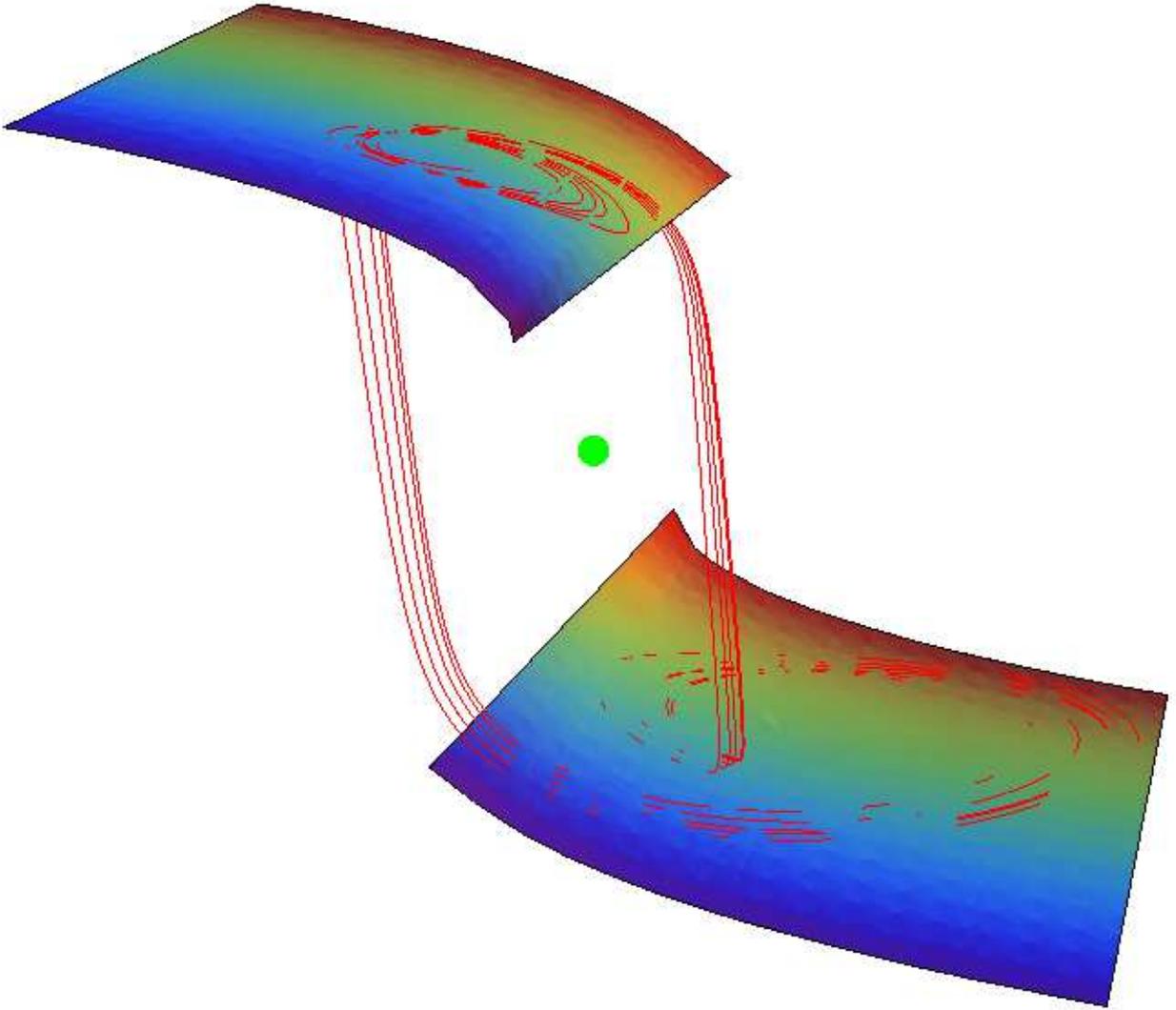}} \caption{Fourth-order Chua's
cubic \textit{invariant manifold} in the ($x_1x_2x_3$) space for:
$\alpha _1 = 2.1429$, $\alpha _2 = - 0.18$, $\beta _1 = 0.0774$,
$\beta _2 = 0.003$, $a = - 0.42$, $b = 1.2.$}
\label{fig4}
\end{figure}

\subsection{Five-dimensional models}

The fifth-order cubic Chua's circuit [Hao \textit{et al.}, 2005] may be described starting
from the same set of differential equations as (\ref{eq28}) but while replacing the
piecewise linear function by a smooth cubic nonlinear.

\begin{equation}
\label{eq36}
\overrightarrow V \left( {{\begin{array}{*{20}c}
 {\frac{dx_1 }{dt}} \hfill \\
 {\frac{dx_2 }{dt}} \hfill \\
 {\frac{dx_3 }{dt}} \hfill \\
 {\begin{array}{l}
 \frac{dx_4 }{dt} \\
 \frac{dx_5 }{dt} \\
 \end{array}} \hfill \\
\end{array} }} \right) = \overrightarrow \Im \left( {{\begin{array}{*{20}c}
 {f_1 \left( {x_1 ,x_3 ,x_3 ,x_4 ,x_5 } \right)} \hfill \\
 {f_2 \left( {x_1 ,x_3 ,x_3 ,x_4 ,x_5 } \right)} \hfill \\
 {f_3 \left( {x_1 ,x_3 ,x_3 ,x_4 ,x_5 } \right)} \hfill \\
 {\begin{array}{l}
 f_4 \left( {x_1 ,x_3 ,x_3 ,x_4 ,x_5 } \right) \\
 f_5 \left( {x_1 ,x_3 ,x_3 ,x_4 ,x_5 } \right) \\
 \end{array}} \hfill \\
\end{array} }} \right) = \left( {{\begin{array}{*{20}c}
 {\alpha _1 \left( {x_2 - x_1 - \hat{k}\left( {x_1 } \right)} \right)} \hfill \\
 {\alpha _2 x_1 - x_2 + x_3 } \hfill \\
 {\beta _1 \left( {x_4 - x_2 } \right)} \hfill \\
 {\begin{array}{l}
 \beta _2 \left( {x_3 + x_5 } \right) \\
 \gamma _2 \left( {x_4 + \gamma _1 x_5 } \right) \\
 \end{array}} \hfill \\
\end{array} }} \right)
\end{equation}

The function $\hat {k}\left( {x_1 } \right)$ describing the electrical
response of the nonlinear resistor is an odd-symmetric function similar to
the piecewise linear nonlinearity $k\left( {x_1 } \right)$ for which the
parameters $c_1 = 0.1068$ and $c_2 = - 0.3056$ are determined while using
least-square method [Tsuneda, 2005] and which characteristics is defined by:

\begin{equation}
\label{eq37}
\hat {k}\left( {x_1 } \right) = c_1 x_1^3 + c_2 x_1
\end{equation}

The real parameters $\alpha _i $, $\beta _i $ and $\gamma _i $
determined by the particular values of the circuit components are:
$\alpha _1 = 9.934$, $\alpha _2 = 1$, $\beta _1 = 14.47$, $\beta _2
= - 406.5$,\\ $\gamma _1 = - 0.0152$, $\gamma _2 = 41000$, $a = -
1.246$, $b = - 0.6724$, $c_1 = 0.1068$, $c_2 = - 0.3056$ and where
the functions $f_i $ are infinitely differentiable with respect to
all $x_i $, and $t$, i.e., are $C^\infty $ functions in a compact E
included in $\mathbb{R}^5$ and with values in $\mathbb{R}$.

\textit{Curvature of the flow} states that the location of the points where the \textit{fourth curvature of the flow}, i.e., the \textit{fourth curvature }of the
\textit{trajectory curves }integral of Chua's fifth-order system vanishes directly provides its \textit{slow}
\textit{invariant} \textit{manifold} analytical equation. According to Proposition 3.1, Eq. (\ref{eq2}) may be written:

\begin{equation}
\label{eq38}
\phi \left( \vec {X} \right) = \overrightarrow V \cdot \left( {\vec {\gamma
} \wedge \dot {\vec {\gamma }} \wedge \ddot {\vec {\gamma }} \wedge \dddot
{\vec {\gamma }} \wedge \ddddot {\vec {\gamma }}} \right) = 0
\end{equation}

Then, it may be proved that in the vicinity of the \textit{singular
approximation} defined by $f_1 \left( \vec {X} \right) = 0$ the
functional jacobian matrix is stationary, i.e., its time derivative
vanishes identically and so, Lie derivative $L_{\vec {V}} \phi
\left( \vec {X} \right) = 0$ vanishes identically. Thus, according
to \textit{Darboux Theorem }[1878], the manifold $\phi \left( \vec
{X} \right)$ which is \textit{locally invariant} is plotted in Fig. 5.

\begin{figure}[htbp] \centerline{\includegraphics{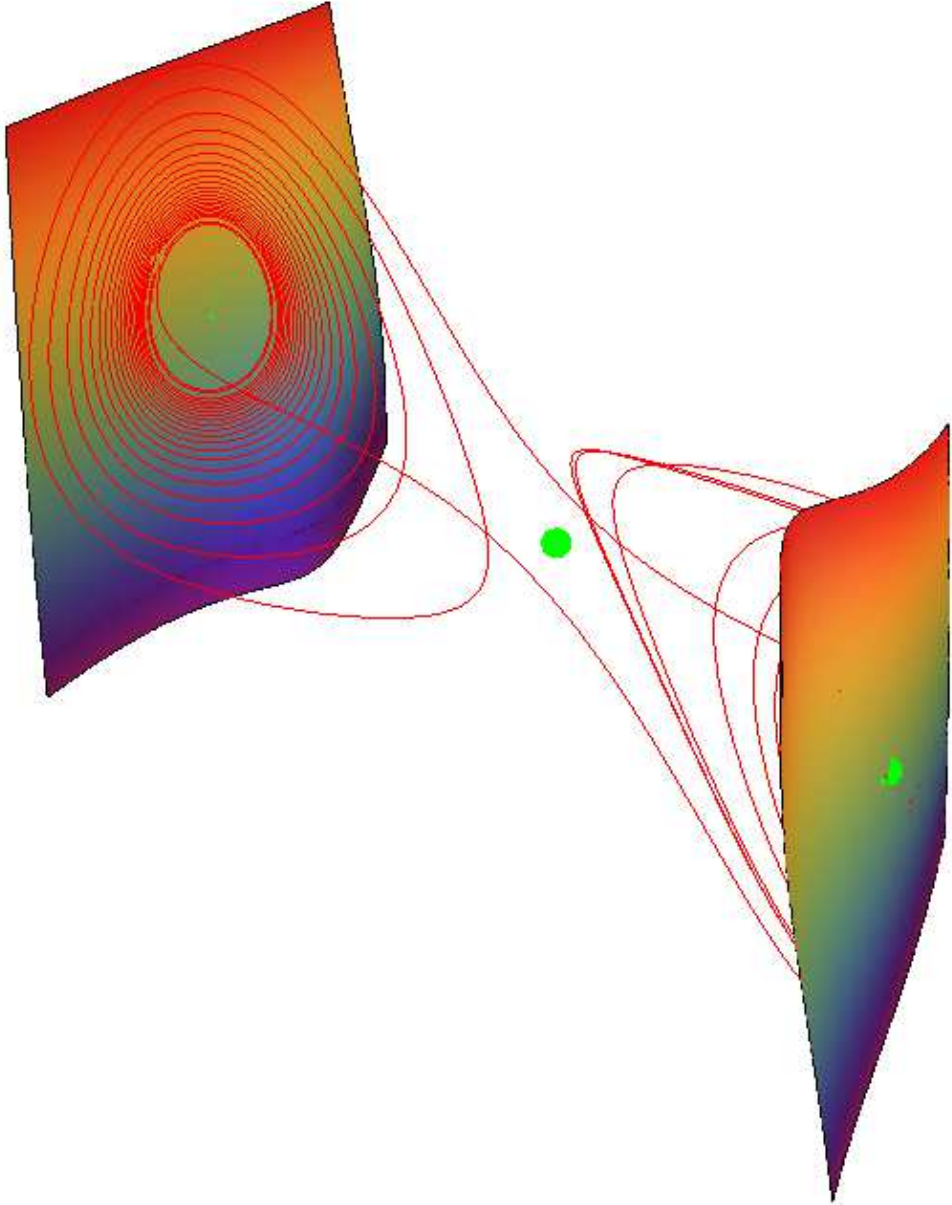}}
\caption{Fifth-order Chua's cubic \textit{invariant manifold} in the
($x_1x_2x_3$) space for: $\alpha _2 = 1$, $\alpha _1 = 9.934$,
$\beta _1 = 14.47$, $\beta _2 = - 406.5$, $\gamma _1 = - 0.0152$,
$\gamma _2 = 41000$, $c_1 = 0.1068$, $c_2 = - 0.3056.$} \label{fig5}
\end{figure}

\section{High-dimensional nonlinear models}

In this section two examples are considered. The former is a
nonlinear fifth-order model of magnetoconvection [Knobloch
\textit{et al.}, 1981] for which the slow invariant manifold will be
directly provided by using \textit{curvature of the flow}. The
latter is an artificial nonlinear fifth-order model [Gear \textit{et
al.}, 2005] having three \textit{attractive invariant manifolds}.

\subsection{Five-dimensional magnetoconvection model}

A fifth-order system for magnetoconvection [Knobloch \textit{et al.}, 1981] is designed to
describe nonlinear coupling between Rayleigh-Bernard convection and an
external magnetic field. This type of system was first presented by Veronis
[Veronis, 1966] in studying a rotating fluid. The fifth-order system of
magnetoconvection is a straightforward extension of the Lorenz model for the
Boussinesq convection interacting with the magnetic field. The fifth-order
autonomous system of magnetoconvection is given as follows:

\begin{equation}
\label{eq39} \overrightarrow V \left( {{\begin{array}{*{20}c}
 {\frac{dx_1 }{dt}} \hfill \\
 {\frac{dx_2 }{dt}} \hfill \\
 {\frac{dx_3 }{dt}} \hfill \\
 {\begin{array}{l}
 \frac{dx_4 }{dt} \\
 \frac{dx_5 }{dt} \\
 \end{array}} \hfill \\
\end{array} }} \right) = \left( {{\begin{array}{*{20}c}
 {\sigma \left[ { - x_1 + rx_2 - qx_4 \left( {1 + \frac{\omega \left( {3 -
\omega } \right)}{\varsigma ^2\left( {4 - \omega } \right)}x_5 } \right)}
\right]} \hfill \\
 { - x_2 + x_1 - x_1 x_3 } \hfill \\
 {\omega \left( { - x_3 + x_1 x_2 } \right)} \hfill \\
 {\begin{array}{l}
 - \varsigma \left( {x_4 - x_1 } \right) - \frac{\omega }{\varsigma \left(
{4 - \omega } \right)}x_1 x_5 \\
 - \varsigma \left( {4 - \omega } \right)\left( {x_5 - x_1 x_4 } \right) \\
 \end{array}} \hfill \\
\end{array} }} \right)
\end{equation}

where $x_1 \left( t \right)$ represents the first-order velocity
perturbation, while$x_2 \left( t \right)$, $x_3 \left( t \right)$, $x_4
\left( t \right)$ and $x_5 \left( t \right)$ are measures of the first- and
the second-order perturbations to the temperature and to the magnetic flux
function, respectively. With the five real parameters where $\varsigma =
0.09683$ is the magnetic \textit{Prandtl} number (the ratio of the magnetic to the thermal
diffusivity), $\sigma = 1$ is the \textit{Prandtl} number, $r = 14.47$ is a normalized
\textit{Rayleigh} number, $q = 5$ is a normalized \textit{Chandrasekhar} number, and $\omega = 0.1081$ is a
geometrical parameter and where the functions $f_i $ are infinitely
differentiable with respect to all $x_i $, and $t$, i.e., are $C^\infty $
functions in a compact E included in $\mathbb{R}^5$ and with values in
$\mathbb{R}$.

\textit{Curvature of the flow} states that the location of the
points where the \textit{fourth curvature of the flow}, i.e., the
\textit{fourth curvature }of the \textit{trajectory curves }integral
of fifth-order magnetoconvection system vanishes directly provides
its \textit{slow} \textit{invariant} \textit{manifold} analytical
equation. According to Proposition 3.1, Eq. (\ref{eq2}) may be
written:

\begin{equation}
\label{eq40}
\phi \left( \vec {X} \right) = \overrightarrow V \cdot \left( {\vec {\gamma
} \wedge \dot {\vec {\gamma }} \wedge \ddot {\vec {\gamma }} \wedge \dddot
{\vec {\gamma }} \wedge \ddddot {\vec {\gamma }}} \right) = 0
\end{equation}

Then, it may be proved that in the vicinity of the \textit{singular
approximation} defined by $f_1 \left( \vec {X} \right) = 0$ the
functional jacobian matrix is stationary, i.e., its time derivative
vanishes identically and so, Lie derivative $L_{\vec {V}} \phi
\left( \vec {X} \right) = 0$ vanishes identically. Thus, according
to \textit{Darboux Theorem }[1878], the manifold $\phi \left( \vec
{X} \right)$ which is \textit{locally invariant} is plotted in Fig. 6.

\begin{figure}[htbp]
\centerline{\includegraphics{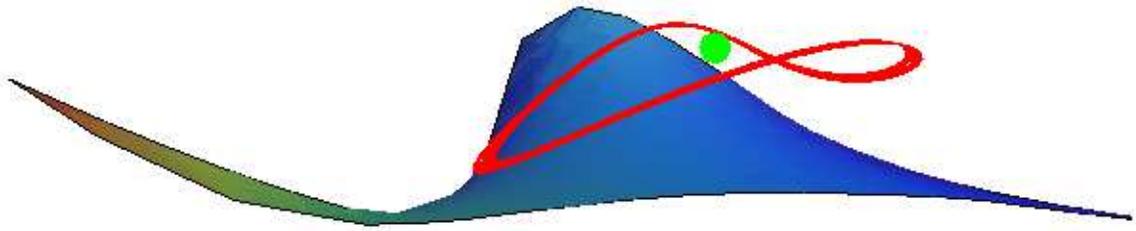}} \caption{Fifth-order
magnetoconvection \textit{invariant manifold} in the ($x_1x_2x_3$)
space for: $\varsigma = 0.09683$, $\sigma = 1$, $r = 14.47$, $q =
5$, $\omega = 0.1081.$} \label{fig6}
\end{figure}

\subsection{Five-dimensional nonlinear model}

Let's consider the following fifth-order nonlinear dynamical system
[Gear \textit{et al.}, 2005]

\begin{equation}
\label{eq41}
\overrightarrow V = \left( {{\begin{array}{*{20}c}
 {\frac{dx_1 }{dt}} \hfill \\
 {\frac{dx_2 }{dt}} \hfill \\
 {\frac{dx_3 }{dt}} \hfill \\
 {\begin{array}{l}
 \frac{dx_4 }{dt} \\
 \frac{dx_5 }{dt} \\
 \end{array}} \hfill \\
\end{array} }} \right) = \left( {{\begin{array}{*{20}c}
 { - x_2 } \hfill \\
 {x_1 } \hfill \\
 {L\left( {x_1^2 + x_2^2 - x_3 } \right)} \hfill \\
 {\begin{array}{l}
 \beta _1 + x_4^2 \\
 \beta _2 + x_2^2 \\
 \end{array}} \hfill \\
\end{array} }} \right)
\end{equation}

where the real parameters values may be arbitrarily chosen as $L = 1000$,
$\beta _1 = 800$, $\beta _2 = 1200$ and where the functions $f_i $ are
infinitely differentiable with respect to all $x_i $, and $t$, i.e., are
$C^\infty $ functions in a compact E included in $\mathbb{R}^5$ and with
values in $\mathbb{R}$.

\textit{Curvature of the flow} states that the location of the
points where the \textit{fourth curvature of the flow}, i.e., the
\textit{fourth curvature }of the \textit{trajectory curves }integral
of fifth-order nonlinear dynamical system vanishes directly provides
its \textit{invariant} \textit{manifolds} analytical equation.
According to Proposition 3.1, Eq. (\ref{eq2}) may be written:

\begin{eqnarray}
\label{eq42} \phi \left( \vec {X} \right) & = & \overrightarrow V
\cdot \left( {\vec {\gamma } \wedge \dot {\vec {\gamma }} \wedge
\ddot {\vec {\gamma }} \wedge \dddot {\vec {\gamma }} \wedge \ddddot
{\vec {\gamma }}} \right)\nonumber \\ & = & \left( {x_1^2 + x_2^2 }
\right)\left( {x_1^2 + x_2^2 - x_3 } \right)\left( {x_4^2 + \beta _1
} \right)Q\left( \vec {X} \right) = 0
\end{eqnarray}

where $Q\left( \vec {X} \right)$ is an irreducible polynomial and
while posing:\\

\[
\varphi \left( \vec {X} \right) = \left( {x_1^2 + x_2^2 }
\right)\left( {x_1^2 + x_2^2 - x_3 } \right)\left( {x_4^2 + \beta _1
} \right)
\]

it may be established that:

\begin{eqnarray}
\label{eq43} L_{\vec {V}} \varphi \left( \vec {X} \right) & = & -
\left( {L - 2x_4 } \right)\left( {x_1^2 + x_2^2 } \right)\left(
{x_1^2 + x_2^2 - x_3 } \right)\left( {x_4^2 + \beta _1 }
\right)\nonumber \\ & = & K\left( \vec {X} \right)\varphi \left(
\vec {X} \right)
\end{eqnarray}

Thus, according to \textit{Darboux Theorem }[1878], the five-dimensional model has three invariant
manifolds, namely $\varphi \left( \vec {X} \right)$ is \textit{invariant}. Moreover, it may be
proved that $\left( {x_1^2 + x_2^2 } \right)$ is first integral. So,
\textit{curvature of the flow} may also be used to ``detect'' first integral of dynamical systems.

\section{Slow invariant manifolds gallery}

In this section two examples of \textit{slow invariant manifolds} of
chaotic attractors are presented. The first (Fig. 7a.) is a chemical kinetics
model used by Gaspard and Nicolis (Journal of Statistics Physics,
Vol. 32, N° 3, 1983, 499 - 518). The second (Fig. 7b) is a neuronal bursting
model elaborated by Hindmarsh-Rose (Philos. Trans. Roy. Soc. London
Ser. B 221, 1984, 87-102).

\begin{figure}[htbp]
\centerline{\includegraphics[width=4.17in,height=4.17in]{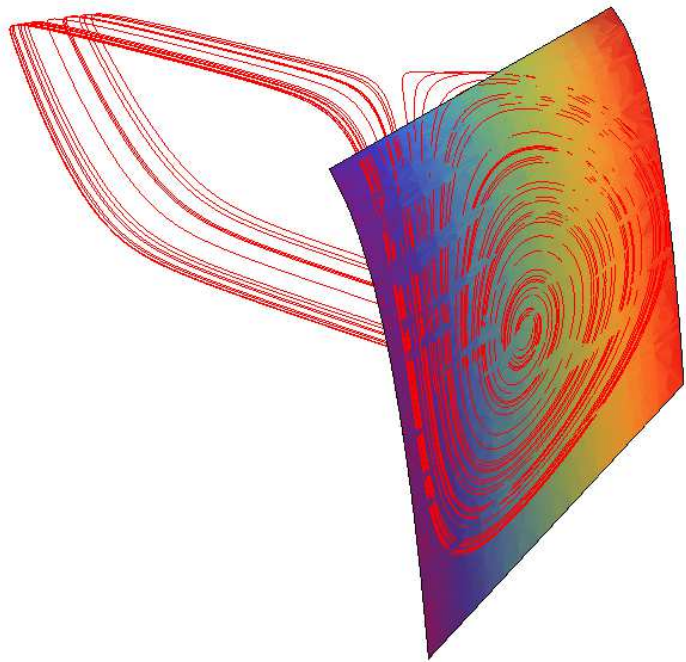}} \caption{Chemical kinetics model}
\label{fig7a}
\end{figure}

\begin{figure}[htbp]
\centerline{\includegraphics[width=4.17in,height=4.17in]{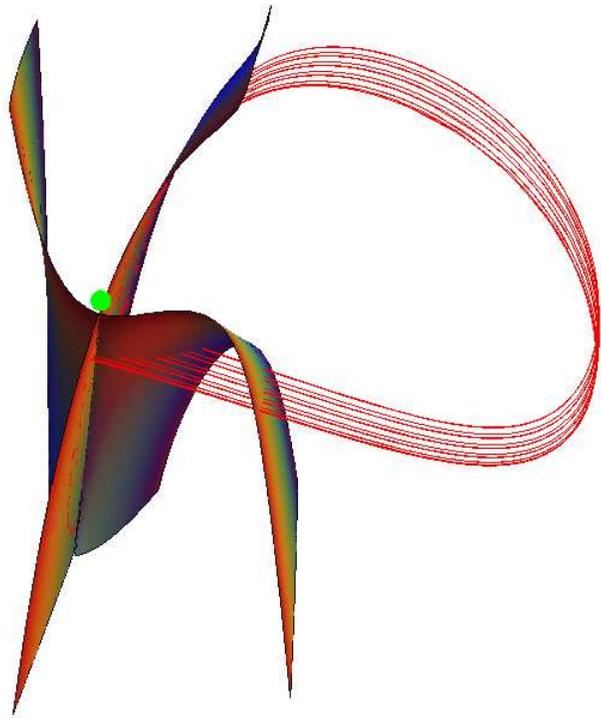}} \caption{Neuronal bursting model}
\label{fig7a}
\end{figure}

A gallery of \textit{slow invariant manifolds} is accessible at:
http://ginoux.univ-tln.fr

\newpage

\section{Discussion}

During the twentieth century various methods have been developed in order to
determine the \textit{slow invariant} \textit{manifold} analytical equation associated to \textit{slow-fast dynamical systems }or \textit{singularly perturbed systems} among which the
so-called \textit{Geometric Singular Perturbation Theory} [Fenichel 1979] and the \textit{Tangent Linear System Approximation} [Rossetto \textit{et al.} 1998]. As pointed out by
O'Malley [1974 p. 78, 1991 p. 21] the problem for finding the \textit{slow invariant manifold} analytical
equation with the \textit{Geometric Singular Perturbation Theory} turned into a regular perturbation problem in which one
generally expected the asymptotic validity of such expansion to breakdown.
Moreover, for high-dimensional \textit{singularly perturbed systems} \textit{slow invariant manifold} analytical equation determination lead to
tedious calculations. The \textit{Tangent Linear System Approximation} the generalization of which is presented in
appendix, provided the \textit{slow manifold} analytical equation of $n$-dimensional dynamical
systems according to the ``slow'' eigenvectors of the \textit{tangent linear system}, i.e., according to
the ``slow'' eigenvalues. Nevertheless, the presence of these eigenvalues
(real or complex conjugated) prevented from expressing this equation
explicitly. Moreover, starting from dimension five \textit{Galois Theory} precludes from
analytically computing eigenvalues associated with the functional jacobian
matrix of a five-dimensional dynamical system.

In this work, while considering \textit{trajectory curves}, integral
of $n$-dimensional dynamical systems, within the framework of
\textit{Differential Geometry} as curves in Euclidean $n$-space it
has be established that the \textit{curvature of the flow}, i.e.,
the \textit{curvature} of the \textit{trajectory curves} of any
$n$-dimensional dynamical system directly provides its \textit{slow
manifold} analytical equation the invariance of which has been
proved according to \textit{Darboux theorem}. Thus, it has been
stated that since \textit{curvature} only involves time derivatives
of the velocity vector field and uses neither eigenvectors nor
asymptotic expansions this simplifying method improves the
\textit{slow invariant manifold }analytical equation determination
of high-dimensional dynamical systems. Chua's paradigmatic models
and nonlinear magnetoconvection high-dimensional dynamical system
have exemplified this result. Since it has been shown in the
appendix that \textit{curvature of the flow} generalizes the
\textit{Tangent Linear System Approximation} and encompasses the
so-called \textit{Geometric Singular Perturbation Theory}, it may be
applied for \textit{slow invariant manifolds} determination of
various kinds of high-dimensional dynamical system such as
\textit{Chemical kinetics}, \textit{Neuronal Bursting models},
\textit{L.A.S.E.R. models}\ldots

Two of the main perspectives to be given at this work may be
highlighted. The former is \textit{bifurcations}. It seems
reasonable to consider that a \textit{bifurcation} would modify the
shape of the manifold and so conversely, geometric interpretations
could enable to highlight such \textit{bifurcations}. And the latter
deals with the particular feature highlighted in Sec. 5.2, i.e.,
that \textit{curvature of the flow} enables ``detecting''
\textit{first integral} of dynamical systems. These works in
progress will be developed in another publication.

\section*{References}

\hspace{0.2in} Andronov, A. A., Khaikin, S. E. {\&} Vitt, A. A.
[1937] \textit{Theory of oscillators}, I, Moscow (Engl. transl.,
Princeton
Univ. Press, Princeton, N. J., 1949).\\

Christopher, C., Llibre, J. {\&} Pereira, J.V. [2007] ``Multiplicity
of invariant algebraic curves in polynomial vector fields,''
\textit{Pac. J. Math.} 229, 63--117.\\

Chua, L. O., Komuro, M. {\&} Matsumoto, T. [1986] ``The Double
Scroll Family,''\textit{ IEEE Trans. Circuits Syst.}, CAS-33
(\ref{eq10}), 1072-1118.\\

Coddington, E.A. {\&} Levinson, N. [1955] \textit{Theory of Ordinary
Differential Equations}, Mac Graw Hill, New York.\\

Cole, J.D. [1968] ``Perturbation Methods in Applied Mathematics,''
Blaisdell, Waltham, MA.\\

Darboux, G. [1878] ``Sur les \'{e}quations diff\'{e}rentielles
alg\'{e}briques du premier ordre et du premier degr\'{e},''
\textit{Bull. Sci. Math. Sr.} 2(\ref{eq2}), 60-96, 123-143,
151-200.\\

Fenichel, N. [1971] ``Persistence and Smoothness of Invariant
Manifolds for Flows,'' \textit{Ind. Univ. Math. J.} 21, 193-225.\\

Fenichel, N. [1974] ``Asymptotic stability with rate conditions,''
\textit{Ind. Univ. Math. J.} 23, 1109-1137.\\

Fenichel, N. [1977] ``Asymptotic stability with rate conditions
II,'' \textit{Ind. Univ. Math. J.} 26, 81-93.\\

Fenichel, N. [1979] ``Geometric singular perturbation theory for
ordinary differential equations,'' \textit{J. Diff. Eq.} 31,
53-98.\\

Frenet, F. [1852] ``Sur les courbes \`{a} double
courbure,''~Th\`{e}se Toulouse, 1847. R\'{e}sum\'{e} dans \textit{J.
de Math.}, 17.\\

Gear, C.W., Kaper, T.J., Kevrekidis, I. {\&} Zagaris, A. [2004]
``Projecting to a slow manifold: singularly perturbed systems and
legacy codes,'' \textit{SIAM Journal on Applied Dynamical Systems},
preprint.\\

Ginoux, J.M. {\&} Rossetto, B. [2006] ``Differential Geometry and
Mechanics Applications to Chaotic Dynamical Systems,'' \textit{Int.
J. Bifurcation and Chaos} 4, Vol. 16, 887-910.\\

Gluck, H. [1966] ``Higher Curvatures of Curves in Euclidean Space,''
\textit{The American Mathematical Monthly}, Vol. 73, No. 7,
699-704.\\

Hao, L., Liu, J. {\&} Wang, R. [2005] ``Analysis of a Fifth-Order
Hyperchaotic Circuit,'' in \textit{IEEE Int. Workshop VLSI Design
{\&} Video Tech.}, 268-271.\\

Knobloch, E. {\&} Proctor, M. [1981] ``Nonlinear periodic convection in
double-diffusive systems,'' \textit{J. Fluid Mech.} 108, 291-316.\\

Levinson, N. [1949] ``A second order differential equation with singular
solutions,'' \textit{Ann. Math.} 50, 127-153.\\

Lichnerowicz, A. [1950] \textit{\'{E}l\'{e}ments de Calcul Tensoriel}, Armand Colin, Paris.\\

Liu X., Wang, J. {\&} Huang, L. [2007] ``Attractors of Fourth-Order
Chua's Circuit and Chaos Control,'' \textit{Int. J. Bifurcation and Chaos} 8, Vol. 17, 2705 -2722.\\

O'Malley, R.E. [1974] \textit{Introduction to Singular Perturbations}, Academic Press, New York.\\

O'Malley, R.E. [1991] \textit{Singular Perturbation Methods for Ordinary Differential Equations}, Springer-Verlag, New York.\\

Poincar\'{e}, H. [1881] ``Sur les courbes d\'{e}finies par une
\'{e}quation
diff\'{e}rentielle,'' \textit{J. Math. Pures et Appl.}, S\'{e}rie III, 7, 375-422.\\

Poincar\'{e}, H. [1882] ``Sur les courbes d\'{e}finies par une
\'{e}quation diff\'{e}rentielle,'' \textit{J. de Math Pures Appl.},
S\'{e}rie III, 8, 251-296.\\

Poincar\'{e}, H. [1885] ``Sur les courbes d\'{e}finies par une
\'{e}quation
diff\'{e}rentielle,'' \textit{J. Math. Pures et Appl.}, S\'{e}rie IV, 1, 167-244.\\

Poincar\'{e}, H. [1886] ``Sur les courbes d\'{e}finies par une \'{e}quation
diff\'{e}rentielle,'' \textit{J. Math. Pures et Appl.}, S\'{e}rie IV, 2, 151-217.\\

Postnikov, M. [1981] \textit{Le\c{c}ons de G\'{e}om\'{e}trie -- Alg\`{e}bre lin\'{e}aire et G\'{e}om\'{e}trie Diff\'{e}rentielle}, Editions Mir, Moscou.\\

Rossetto, B. [1993] ``Chua's circuit as a slow-fast autonomous dynamical.
system,'' \textit{J. Circuits Syst. Comput.} 3(\ref{eq2}), 483-496.\\

Rossetto, B., Lenzini, T., Ramdani, S. {\&} Suchey, G. [1998] ``Slow-fast
autonomous dynamical systems,'' \textit{Int. J. Bifurcation and Chaos}, 8, Vol. 11, 2135-2145.\\

Schlomiuk, D. [1993] ``Elementary first integrals of differential equations
and invariant algebraic curves,'' \textit{Expositiones Mathematicae}, 11, 433-454.\\

Thamilmaran, K., Lakshmanan M. {\&} Venkatesan A. [2004]
``Hyperchaos in. a
modified canonical Chua's circuit,'' \textit{Int. J. Bifurcation and Chaos}, vol. 14, 221-243.\\

Tikhonov, N. [1948] ``On the dependence of solutions of differential
equations on a small parameter,''\textit{ Mat. Sb.}, 22 : 2, 193--204 (In Russian).\\

Tsuneda, A. [2005] ``A gallery of attractors from smooth Chua's equation,''
\textit{Int. J. Bifurcation and Chaos}, vol. 15, 1-49.\\

Veronis, G. [1966] ``Motions at subcritical values of the Rayleigh number in
a rotating fluid,'' \textit{J. Fluid Mech.}, Vol. 24, 545--554.\\

Wasow, W.R. [1965] \textit{Asymptotic Expansions for Ordinary Differential Equations}, Wiley-Interscience, New York.\\

\newpage

\section*{Appendix}

\appendix

The aim of this appendix is to present definitions inherent to
\textit{Differential Geometry} such as the concept of$
n$-dimensional\textit{ smooth curves}, \textit{generalized}
\textit{Fr\'{e}net frame}, Gram-Schmidt orthogonalization process
for computing \textit{curvatures} of \textit{trajectory curves} in
Euclidean $n$-space as well as proofs of identities (A.10, A.15 {\&}
A.16) used in this work. Then, it is established that
\textit{curvature of the flow} for \textit{slow invariant manifold}
analytical equation determination of high-dimensional dynamical
systems generalizes on the one hand the \textit{tangent linear
system} \textit{approximation }[Rossetto \textit{et al.}, 1998] and
encompasses on the other hand the so-called \textit{Geometric
Singular Perturbation Theory} [Fenichel, 1979].

\section{Differential Geometry}

Within the framework of \textit{Differential Geometry}, $n$-dimensional \textit{smooth} \textit{curves}, i.e., \textit{smooth curves} in Euclidean $n-$space are
defined by a \textit{regular parametric representation} in \textit{terms of arc length} also called \textit{natural representation} or \textit{unit speed parametrization}. According to Herman Gluck [1966] local
metrics properties of \textit{curvatures} may be directly deduced from \textit{curves parametrized in terms of time} and so \textit{natural representation} is not
necessary.

\subsection{Concept of curves}

Considering \textit{trajectory curve} $\vec {X}\left( t \right)$ integral of a $n$-dimensional dynamical
system (\ref{eq1}) as ``the motion of a variable point in a space of dimension $n$''
leads to the following definition.

\begin{de}

A \textit{smooth} \textit{parametrized}\footnote{ with any kind of
parametrization.} \textit{curve} in $\mathbb{R}^n$ is a
\textit{smooth} map $\vec {X}\left( t \right):\left[ {a,b} \right]
\to \mathbb{R}^n$ from a closed interval $\left[ {a,b} \right]$ into
$\mathbb{R}^n$. A map is said to be \textit{smooth} or
\textit{infinitely many times differentiable} if the coordinate
functions $x_1 ,x_2 ,\ldots ,x_n $ of $\vec {X} = \left[ {x_1 ,x_2
,\ldots ,x_n } \right]^t$ have continuous partial derivatives of any
order.

\end{de}

\subsection{Gram-Schmidt process and Fr\'{e}net moving frame}

There are many moving frames along a \textit{trajectory curve} and
most of them are not related to local metrics properties of
\textit{curvatures}. This is not the case for Fr\'{e}net frame
[1852]. In this sub-section generalized Fr\'{e}net frame for
$n$-dimensional \textit{trajectory curves} in Euclidean $n$-space is
recalled.\\

Let's suppose that the \textit{trajectory curve} $\vec {X}\left( t
\right)$, \textit{parametrized in terms of time}, is of\textit{
general type }in $\mathbb{R}^n$, i.e., that the first $n - 1$ time
derivatives: $\dot {\vec {X}}\left( t \right)$, $\ddot {\vec
{X}}\left( t \right)$, \ldots , $\mathop {\vec {X}}\limits^{\left(
{n - 1} \right)} \left( t \right)$, are linearly independent for all
$t$.\\

A moving frame along a \textit{trajectory curve} $\vec {X}\left( t \right)$ of \textit{general type }in $\mathbb{R}^n$ is a
collection of $i$ vectors $\vec {u}_1 \left( t \right)$, $\vec {u}_2 \left(
t \right)$, \ldots , $\vec {u}_i \left( t \right)$ along $\vec {X}\left( t
\right)$ forming an \textit{orthogonal basis}, such that:

\begin{equation}
\label{eq44}
\vec {u}_i \left( t \right) \cdot \vec {u}_j \left( t \right) = 0
\end{equation}

for all $t$ and for $i \ne j$. These vectors $\vec {u}_i \left( t \right)$
may be determined by application of the Gram-Schmidt orthogonalization
process described below.

\textbf{Gram-Schmidt process. }Let $\dot {\vec {X}}\left( t
\right)$, $\ddot {\vec {X}}\left( t \right)$, \ldots , $\mathop
{\vec {X}}\limits^{\left( {n - 1} \right)} \left( t \right)$ be
linearly independent vectors for all $t$ in $\mathbb{R}^n$.
According to Gram-Schmidt process [Lichnerowicz, 1950 p. 30, Gluck,
1966] the vectors $\vec {u}_1 \left( t \right)$, $\vec {u}_2 \left(
t \right)$, \ldots , $\vec {u}_i \left( t \right)$ forming an
orthogonal basis are defined by:

\[
\vec {u}_1 \left( t \right) = \dot {\vec {X}}\left( t \right)
\]

\[
\vec {u}_2 \left( t \right) = \ddot {\vec {X}}\left( t \right) - \left(
{\frac{\vec {u}_1 \left( t \right) \cdot \ddot {\vec {X}}\left( t
\right)}{\vec {u}_1 \left( t \right) \cdot \vec {u}_1 \left( t \right)}}
\right)\vec {u}_1 \left( t \right)
\]

\[
\vec {u}_3 \left( t \right) = \dddot {\vec {X}}\left( t \right) - \left(
{\frac{\vec {u}_1 \left( t \right) \cdot \dddot {\vec {X}}\left( t
\right)}{\vec {u}_1 \left( t \right) \cdot \vec {u}_1 \left( t \right)}}
\right)\vec {u}_1 \left( t \right) - \left( {\frac{\vec {u}_2 \left( t
\right) \cdot \dddot {\vec {X}}\left( t \right)}{\vec {u}_2 \left( t \right)
\cdot \vec {u}_2 \left( t \right)}} \right)\vec {u}_2 \left( t \right)
\]

\begin{center}
\ldots \ldots \ldots \ldots \ldots \ldots \ldots \ldots \ldots \ldots \ldots
\ldots \ldots \ldots \ldots \ldots \ldots \ldots \ldots \ldots .\ldots
\ldots \ldots
\end{center}

\begin{equation}
\label{eq45} \vec {u}_i \left( t \right) = \mathop {\vec
{X}}\limits^{\left( i \right)} \left( t \right) - \sum\limits_{j =
1}^{n - 1} {\left( {\frac{\vec {u}_j \left( t \right) \cdot \mathop
{\vec {X}}\limits^{\left( i \right)} \left( t \right)}{\vec {u}_j
\left( t \right) \cdot {\vec {u}_j} \left( t \right)}} \right)\vec
{u}_j \left( t \right)}
\end{equation}

\textbf{Generalized Fr\'{e}net moving frame. }Starting from the
vectors $\vec {u}_1 \left( t \right)$, $\vec {u}_2 \left( t
\right)$, \ldots , $\vec {u}_i \left( t \right)$ forming an
orthogonal basis, \textit{generalized Fr\'{e}net moving frame} for
the \textit{trajectory curve} $\vec {X}\left( t \right)$ of
\textit{general type }in $\mathbb{R}^n$ may be built. Thus
derivation with respect to time $t$ leads to the \textit{generalized
Fr\'{e}net formulas }in Euclidean $n$-space:

\begin{equation}
\label{eq46}
\dot {\vec {u}}_i \left( t \right) = v\sum\limits_{j = 1}^n {\alpha _{ij}
\vec {u}_j \left( t \right)}
\end{equation}

with $i = 1,2,\ldots ,n$ and where $v = \left\| \dot {\vec {X}} \right\| =
\left\| {\overrightarrow V } \right\|$ represents the Euclidean norm of the
velocity vector field. Moreover, according to Eq. (\ref{eq44}) implies that:

\begin{equation}
\label{eq47}
\dot {\vec {u}}_i \left( t \right) \cdot \vec {u}_j \left( t \right) + \vec
{u}_i \left( t \right) \cdot \dot {\vec {u}}_j \left( t \right) = 0
\end{equation}

So, $\alpha _{ii} = 0$ and $\alpha _{ij} = 0$ for $j < i - 1$. Thus, only
$\alpha _{i,i + 1} = - \alpha _{i + 1,i} $ are not identically zero.

Let's pose:

\begin{equation}
\label{eq48}
\kappa _1 = \alpha _{12} ,
\quad
\kappa _2 = \alpha _{23} , \quad \ldots ,
\quad
\kappa _{n - 1} = \alpha _{n - 1,n}
\end{equation}

The \textit{generalized Fr\'{e}net formulas} associated with a \textit{trajectory curve} in Euclidean $n$-space read:

\begin{equation}
\label{eq49}
\left\{ {{\begin{array}{*{20}c}
 {\dot {\vec {u}}_1 \left( t \right) = v\kappa _1 \vec {u}_2 \left( t
\right)} \\
 {\dot {\vec {u}}_2 \left( t \right) = v\left[ { - \kappa _1 \vec {u}_1
\left( t \right) + \kappa _2 \vec {u}_3 \left( t \right)} \right]} \\
 {\dot {\vec {u}}_3 \left( t \right) = - v\kappa _2 \vec {u}_2 \left( t
\right)} \\
 { \cdots \cdots \cdots \cdots \cdots \cdots \cdots } \\
 {\dot {\vec {u}}_{n - 1} \left( t \right) = v\left[ { - \kappa _{n - 2}
\vec {u}_{n - 2} \left( t \right) + \kappa _{n - 1} \vec {u}_n \left( t
\right)} \right]} \\
 {\dot {\vec {u}}_n \left( t \right) = - v\kappa _{n - 1} \vec {u}_{n - 1}
\left( t \right)} \\
\end{array} }} \right.
\end{equation}

The functions $\kappa _1 $, $\kappa _2 $, \ldots , $\kappa _{n - 1} $ are
called \textit{curvatures} of \textit{trajectory curve} $\vec {X}\left( t \right)$ of \textit{general type }in $\mathbb{R}^n$ and $\kappa _{n -
1} $ is analogous to the \textit{torsion}.

Thus, according to Gluck [1966, p. 702] \textit{curvatures} of \textit{trajectory curves }$\vec {X}\left( t \right)$
integral of any $n$-dimensional dynamical systems (\ref{eq1}) may be defined by:

\begin{equation}
\label{eq50}
\kappa _i = \frac{\left\| {\vec {u}_{i + 1} \left( t \right)}
\right\|}{\left\| {\vec {u}_1 \left( t \right)} \right\|\left\| {\vec {u}_i
\left( t \right)} \right\|}
\end{equation}

Since $1 \leqslant i \leqslant n - 1$ a $n$-dimensional \textit{trajectory} \textit{curve} has $\left( {n - 1}
\right)$ \textit{curvatures}.

\subsection{Fr\'{e}net trihedron and curvatures of space curves}

\textbf{Fr\'{e}net trihedron. }While normalizing the basis vectors
$\vec {u}_1 \left( t \right)$, $\vec {u}_2 \left( t \right)$, \ldots
, $\vec {u}_n \left( t \right)$ obtained with the Gram-Schmidt
process, the so-called Fr\'{e}net trihedron for \textit{space
curves} may be deduced.\\
Hence, it may be stated that: $\left(
{\frac{\vec {u}_1 \left( t \right)}{\left\| {\vec {u}_1 \left( t
\right)} \right\|},\frac{\vec {u}_2 \left( t \right)}{\left\| {\vec
{u}_2 \left( t \right)} \right\|},\frac{\vec {u}_3 \left( t
\right)}{\left\| {\vec {u}_3 \left( t \right)} \right\|}} \right) =
\left( {\vec {\tau },\vec {n},\vec {b}} \right)$ where $\vec {\tau
}$, $\vec {n}$ and $\vec {b}$ are respectively the tangent, normal
and binormal unit vectors.

Let's notice that the three first time derivatives: $\dot {\vec
{X}}\left( t \right)$, $\ddot {\vec {X}}\left( t \right)$ and
$\dddot {\vec {X}}\left( t \right)$ represent respectively the
velocity, acceleration and over-acceleration vector field namely:
$\overrightarrow V \left( t \right)$, $\vec {\gamma }\left( t
\right)$ and $\dot {\vec {\gamma }}\left( t \right)$. Thus, from the
\textit{generalized Fr\'{e}net formulas} (\ref{eq49}) and Gluck
formulae (\ref{eq50}) of \textit{curvatures}, the first and second
curvatures of \textit{space curves}, i.e., \textit{curvature} and
\textit{torsion} may be found again.\\

\textbf{First curvature.} While replacing basis vectors $\vec {u}_1 \left( t
\right)$ and $\vec {u}_2 \left( t \right)$ resulting from the Gram-Schmidt
process in formulae (\ref{eq50}), (first) \textit{curvature} of \textit{space trajectory curves} is given by:

\begin{equation}
\label{eq51}
\kappa _1 \left( t \right) = \frac{\left\| {\vec {u}_2 \left( t \right)}
\right\|}{\left\| {\vec {u}_1 \left( t \right)} \right\|^2} = \frac{\left\|
{\vec {\gamma }\left( t \right) \wedge \overrightarrow V \left( t \right)}
\right\|}{\left\| {\overrightarrow V } \right\|^3}
\end{equation}

\begin{proof}

While using the Lagrange identity it may be established that:\\

$\left\| {\vec {u}_1 } \right\|^2\left\| {\vec {u}_2 } \right\|^2 =
\left\| {\dot {\vec {X}} \wedge \ddot {\vec {X}}} \right\|^2$. So,
\textit{curvature} $\kappa _1 $ reads:

\[
\kappa _1 \left( t \right) = \frac{\left\| {\vec {u}_2 \left( t \right)}
\right\|}{\left\| {\vec {u}_1 \left( t \right)} \right\|^2} = \frac{\left\|
{\vec {\gamma }\left( t \right) \wedge \overrightarrow V \left( t \right)}
\right\|}{\left\| {\overrightarrow V } \right\|^3}
\]

\end{proof}

\textbf{Second curvature.} While replacing basis vectors $\vec {u}_1 \left(
t \right)$, $\vec {u}_2 \left( t \right)$ and $\vec {u}_3 \left( t \right)$
resulting from the Gram-Schmidt process in formulae (\ref{eq50}), (second)
\textit{curvature}, i.e., \textit{torsion} of \textit{space trajectory curves} is given by:

\begin{equation}
\label{eq52}
\kappa _2 \left( t \right) = \frac{\left\| {\vec {u}_3 \left( t \right)}
\right\|}{\left\| {\vec {u}_1 \left( t \right)} \right\|\left\| {\vec {u}_2
\left( t \right)} \right\|} = - \frac{\dot {\vec {\gamma }}\left( t \right)
\cdot \left( {\vec {\gamma }\left( t \right) \wedge \overrightarrow V \left(
t \right)} \right)}{\left\| {\vec {\gamma }\left( t \right) \wedge
\overrightarrow V \left( t \right)} \right\|^2}
\end{equation}

\begin{proof}

Still using the Lagrange identity, i.e., $\left\| {\vec {u}_1 }
\right\|^2\left\| {\vec {u}_2 } \right\|^2 = \left\| {\dot {\vec
{X}} \wedge \ddot {\vec {X}}} \right\|^2$\textit{ torsion} $\kappa
_2 $ reads:

\[
\kappa _2 = \frac{\left\| {\vec {u}_3 \left( t \right)} \right\|}{\left\|
{\vec {u}_1 \left( t \right)} \right\|\left\| {\vec {u}_2 \left( t \right)}
\right\|} = - \frac{\dot {\vec {\gamma }} \cdot \left( {\vec {\gamma }
\wedge \overrightarrow V } \right)}{\left\| {\vec {\gamma } \wedge
\overrightarrow V } \right\|^2}
\]

\end{proof}

\subsection{Identities proofs }

\textbf{Identity A.10.}

\begin{equation}
\label{eq53} \left[ {\dot {\vec {X}},\ddot {\vec {X}},\ldots
,\mathop {\vec{X}}\limits^{\left( n \right)} } \right] = \dot {\vec
{X}} \cdot \left( {\ddot {\vec {X}} \wedge \dddot {\vec {X}} \wedge
\ldots \wedge \mathop {\vec {X}}\limits^{\left( n \right)} } \right)
= \left\| {\vec {u}_1 } \right\|\left\| {\vec {u}_2 } \right\|\ldots
\left\| {\vec {u}_n } \right\|
\end{equation}

\begin{proof}

According to Postnikov [1981, p. 215], Gram-Schmidt process can be
written

\begin{equation}
\label{eq54} \vec {u}_n \left( t \right) = \sum\limits_{i = 1}^n
{\beta _{ni} \mathop {\vec {X}}\limits^{\left( i \right)} \left( t
\right)}
\end{equation}

Comparing (\ref{eq54}) with (\ref{eq45}) leads to:

\begin{equation}
\label{eq55}
\beta _{ii} = 1
\end{equation}

Using (\ref{eq54}) and (\ref{eq55}), the \textit{inner product} $\vec {u}_1 \cdot \left( {\vec {u}_2 \wedge
\ldots \wedge \vec {u}_n } \right)$ reads:

\begin{equation}
\label{eq56} \vec {u}_1 \cdot \left( {\vec {u}_2 \wedge \ldots
\wedge \vec {u}_n } \right) = \dot {\vec {X}} \cdot \left( {\ddot
{\vec {X}},\ldots ,\mathop {\vec {X}}\limits^{\left( n \right)} }
\right)
\end{equation}

But, since Gram-Schmidt basis is orthogonal, the \textit{inner product} $\vec {u}_1 \cdot \left(
{\vec {u}_2 \wedge \ldots \wedge \vec {u}_n } \right)$ reads too:

\begin{equation}
\label{eq57}
\vec {u}_1 \cdot \left( {\vec {u}_2 \wedge \ldots \wedge \vec {u}_n }
\right) = \left\| {\vec {u}_1 } \right\|\left\| {\vec {u}_2 } \right\|
\cdots \left\| {\vec {u}_n } \right\|
\end{equation}

From (\ref{eq56}) and (\ref{eq57}) it follows that: $\dot {\vec {X}}
\cdot \left( {\ddot {\vec {X}},\ldots ,\mathop {\vec
{X}}\limits^{\left( n \right)} } \right) = \left\| {\vec {u}_1 }
\right\|\left\| {\vec {u}_2 } \right\| \cdots \left\| {\vec {u}_n }
\right\|$.

For example, while omitting the \textit{time} variable the three first Gram-Schmidt
vectors read:

\begin{table}[htbp]
\begin{center}
\begin{tabular}
{|c|}
\hline
$ \mbox{} $\\
$\vec {u}_1 = \beta _{11} \dot {\vec {X}}$\\
$ \mbox{}$\\
\hline
$ \mbox{ } $\\
$\vec {u}_2 = \beta _{21} \dot {\vec {X}} + \beta _{22} \ddot {\vec {X}}$ \\
$ \mbox{ } $\\
\hline
$ \mbox{ } $\\
$\vec {u}_3 = \beta _{31} \dot {\vec {X}} + \beta _{32} \ddot {\vec {X}} + \beta _{33} \dddot {\vec {X}}$ \\
$ \mbox{ } $\\
\hline
\end{tabular}
\label{tab1}
\end{center}
\end{table}

Using (\ref{eq54}) and (\ref{eq55}), the \textit{inner product} $\vec {u}_1 \cdot \left( {\vec {u}_2 \wedge
\vec {u}_3 } \right)$ reads:

\[
\vec {u}_1 \cdot \left( {\vec {u}_2 \wedge \vec {u}_3 } \right) = \beta
_{11} \beta _{22} \beta _{33} \dot {\vec {X}} \cdot \left( {\ddot {\vec {X}}
\wedge \dddot {\vec {X}}} \right) = \left\| {\vec {u}_1 } \right\|\left\|
{\vec {u}_2 } \right\|\left\| {\vec {u}_3 } \right\|
\]

\end{proof}

\textbf{Identity A.15.}

\begin{equation}
\label{eq58}
J\vec {a}_1 .\left( {J\vec {a}_2 \wedge \ldots \wedge J\vec {a}_n } \right)
= Det\left( J \right)\vec {a}_1 .\left( {\vec {a}_2 \wedge \ldots \wedge
\vec {a}_n } \right)
\end{equation}

\begin{proof}

Equation (\ref{eq58}) may also be written with \textit{inner
product}:

\[
J\vec {a}_1 .\left( {J\vec {a}_2 \wedge \ldots \wedge J\vec {a}_n } \right)
= \left[ {J\vec {a}_1 ,J\vec {a}_2 ,\ldots ,J\vec {a}_n } \right] =
Det\left( {J\vec {a}_1 ,J\vec {a}_2 ,\ldots ,J\vec {a}_n } \right)
\]

But, since $\left( {J\vec {a}_1 ,J\vec {a}_2 ,\ldots ,J\vec {a}_n } \right)
= J\left( {\vec {a}_1 ,\vec {a}_2 ,\ldots ,\vec {a}_n } \right)$ and while
using determinant product property, i.e., determinant of the product is
equal to the product of the determinants we have:

\begin{eqnarray}
J\vec {a}_1 .\left( {J\vec {a}_2 \wedge \ldots \wedge J\vec {a}_n }
\right) & = & \left[ {J\vec {a}_1 ,J\vec {a}_2 ,\ldots ,J\vec {a}_n
} \right] \nonumber \\ & = & Det\left( {J\vec {a}_1 ,J\vec {a}_2
,\ldots ,J\vec {a}_n } \right)\nonumber \\ & = &  Det\left( J
\right)Det\left( {\vec {a}_1 ,\vec {a}_2 ,\ldots ,\vec {a}_n }
\right)
\end{eqnarray}

\end{proof}

\textbf{Identity A.16.}

\begin{eqnarray}
\label{eq59} J\vec {a}_1 .\left( {\vec {a}_2 \wedge \ldots \wedge
\vec {a}_n } \right) + \vec {a}_1 .\left( {J\vec {a}_2 \wedge \ldots
\wedge \vec {a}_n } \right) + & \ldots & \nonumber \\
+ \vec {a}_1 .\left( {\vec {a}_2 \wedge \ldots \wedge J\vec {a}_n }
\right) = Tr\left( J \right)\vec {a}_1 .\left( {\vec {a}_2 \wedge
\ldots \wedge \vec {a}_n } \right)
\end{eqnarray}

\begin{proof}

The proof is based on Trace properties such as \textit{linearity}
and \textit{similarity-invariant}.

\end{proof}

\section{Tangent linear system approximation}

\subsection{Assumptions}

The \textit{generalized tangent linear system} \textit{approximation
}requires that the dynamical system (\ref{eq1}) satisfies the
following assumptions:\\

\textbf{(H}$_{1}$\textbf{)} The components $f_i $, of the velocity
vector field $\overrightarrow \Im \left( \vec {X} \right)$ defined in E are
continuous, $C^\infty $ functions in E and with values included in
$\mathbb{R}$.

\textbf{(H}$_{2}$\textbf{)} The dynamical system (\ref{eq1}) satisfies the
\textit{nonlinear part condition} [Rossetto \textit{et al.}, 1998], i.e., that the influence of the nonlinear part of the
Taylor series of the velocity vector field $\overrightarrow \Im \left( \vec
{X} \right)$ of this system is overshadowed by the fast dynamics of the
linear part.

\begin{equation}
\label{eq60}
\overrightarrow \Im \left( \vec {X} \right) = \overrightarrow \Im \left(
{\vec {X}_0 } \right) + \left( {\vec {X} - \vec {X}_0 } \right)\left.
{\frac{d\overrightarrow \Im \left( \vec {X} \right)}{d\vec {X}}}
\right|_{\vec {X}_0 } + O\left( {\left( {\vec {X} - \vec {X}_0 } \right)^2}
\right)
\end{equation}

\textbf{(H}$_{3}$\textbf{)} The functional jacobian matrix associated to
dynamical system (\ref{eq1}) has at least a ``fast'' eigenvalue $\lambda _1 $, i.e.,
with the largest absolute value of the real part.

\newpage

\subsection{Corollaries}

To the dynamical system (\ref{eq1}) is associated a \textit{tangent linear system }defined as follows:

\begin{equation}
\label{eq61}
\frac{d\delta \vec {X}}{dt} = J\left( {\vec {X}_0 } \right)\delta \vec {X}
\end{equation}

where

$\delta \vec {X} = \vec {X} - \vec {X}_0 ,
\quad
\vec {X}_0 = \vec {X}\left( {t_0 } \right)$and $\left.
{\frac{d\overrightarrow \Im \left( \vec {X} \right)}{d\vec {X}}}
\right|_{\vec {X}_0 } = J\left( {\vec {X}_0 } \right)$

\textbf{Corollary B.1.}

The\textit{ nonlinear part condition }implies that the velocity
varies slowly in the vicinity of the \textit{slow manifold}. This
involves that the functional jacobian $J\left( {\vec {X}_0 }
\right)$ varies slowly with time, i.e.,~

\begin{equation}
\label{eq62}
\frac{dJ}{dt}\left( {\vec {X}_0 } \right) = 0
\end{equation}

The solution of the \textit{tangent linear} \textit{system} (\ref{eq61}) is written:

\begin{equation}
\label{eq63}
\delta \vec {X} = e^{J\left( {\vec {X}_0 } \right)\left( {t - t_0 }
\right)}\delta \vec {X}\left( {t_0 } \right)
\end{equation}

So,

\begin{equation}
\label{eq64}
\delta \vec {X} = \sum\limits_{i = 1}^n {a_i } \overrightarrow {Y_{\lambda
_i } }
\end{equation}

where $n$ is the dimension of the eigenspace, $a_{i}$ represents
coefficients depending explicitly on the co-ordinates of space and
implicitly on time and $\overrightarrow {Y_{\lambda _i } } $ the
eigenvectors associated in the functional jacobian of the
\textit{tangent linear system}.\\

\textbf{Corollary B.2.}

In the vicinity of the \textit{slow manifold }the velocity of the dynamical system (\ref{eq1}) and that of
the \textit{tangent linear system} (\ref{eq61}) merge.

\begin{equation}
\label{eq65}
\frac{d\delta \vec {X}}{dt} = \overrightarrow V _T \approx \overrightarrow V
\end{equation}

where $\overrightarrow V _T $ represents the velocity vector associated with
the \textit{tangent linear system}.

The \textit{tangent linear system approximation} consists in spreading the velocity vector field $\overrightarrow V $ on
the eigenbasis associated to the functional jacobian matrix of the \textit{tangent linear system}.

While taking account of (\ref{eq61}) and (\ref{eq64}) we have according to (\ref{eq65}):

\begin{equation}
\label{eq66}
\frac{d\delta \vec {X}}{dt} = J\left( {\vec {X}_0 } \right)\delta \vec {X} =
J\left( {\vec {X}_0 } \right)\sum\limits_{i = 1}^n {a_i } \overrightarrow
{Y_{\lambda _i } } = \sum\limits_{i = 1}^n {a_i } J\left( {\vec {X}_0 }
\right)\overrightarrow {Y_{\lambda _i } } = \sum\limits_{i = 1}^n {a_i }
\lambda _i \overrightarrow {Y_{\lambda _i } }
\end{equation}

Thus, Corollary B.2 provides:

\begin{equation}
\label{eq67}
\frac{d\delta \vec {X}}{dt} = \overrightarrow V _T \approx \overrightarrow V
= \sum\limits_{i = 1}^n {a_i } \lambda _i \overrightarrow {Y_{\lambda _i } }
\end{equation}

Then, existence of an evanescent mode in the vicinity of the \textit{slow manifold} implies
according to Tikhonov's theorem [1952] that $a_1 \lambda _1 \ll 1$. So, the
\textit{coplanarity} condition (\ref{eq67}) provides the \textit{slow manifold} equation of a $n$-dimensional dynamical system
(\ref{eq1}).

\begin{prop}

The coplanarity condition between the velocity vector field
$\overrightarrow V$ of a n-dimensional dynamical system and the slow
eigenvectors $\overrightarrow {Y_{\lambda _i } } $ associated to the
slow eigenvalues $\lambda _i $ of its functional jacobian provides
the slow manifold equation of such system.

\begin{equation}
\label{eq68}
\overrightarrow V = \sum\limits_{i = 2}^n {a_i \overrightarrow {Y_{\lambda
_i } } } = a_2 \overrightarrow {Y_{\lambda _2 } } + \ldots + a_n
\overrightarrow {Y_{\lambda _n } }
\quad
 \Leftrightarrow
\quad
\phi \left( \vec {X} \right) = \overrightarrow V .\left( {\overrightarrow
{Y_{\lambda _2 } } \wedge \ldots \wedge \overrightarrow {Y_{\lambda _n } } }
\right) = 0
\end{equation}

\end{prop}

An alternative proposed by Rossetto \textit{et al.} [1998] uses the ``fast'' eigenvector on
the left associated with the ``fast'' eigenvalue of the transposed
functional jacobian of the \textit{tangent linear system}. In this case the velocity vector field
$\overrightarrow V $ is then orthogonal with the ``fast'' eigenvector on the
left. This \textit{orthogonality} condition also provides the \textit{slow manifold} equation of a $n$-dimensional
dynamical system (\ref{eq1}).

\begin{prop}

The orthogonality condition between the velocity vector field
$\overrightarrow V $ of a n-dimensional dynamical system and the
fast eigenvector ${ }^t\overrightarrow {Y_{\lambda _1 } } $ on the
left associated with the fast eigenvalue $\lambda _1 $ of its
transposed functional jacobian provides the slow manifold equation
of such system.

\begin{equation}
\label{eq69}
\phi \left( \vec {X} \right) = \mbox{ }\overrightarrow V \cdot {
}^t\overrightarrow {Y_{\lambda _1 } } = 0
\end{equation}

\end{prop}

\begin{prop}

Both coplanarity and orthogonality conditions providing the slow
manifold equation are equivalent.

\end{prop}

While using the following identity the proof is obvious:

\begin{equation}
\label{eq70}
\left( {\overrightarrow {Y_{\lambda _2 } } \wedge \overrightarrow
{Y_{\lambda _3 } } \wedge \ldots \wedge \overrightarrow {Y_{\lambda _n } } }
\right) = { }^t\overrightarrow {Y_{\lambda _1 } }
\end{equation}

Thus, \textit{coplanarity} and \textit{orthogonality} conditions are
completely equivalent.\\

Since for low-dimensional two and three dynamical systems the proof has been
already established [Ginoux \textit{et al.}, 2006] while using the \textit{Tangent Linear System Approximation}, for high-dimensional
dynamical systems it may be deduced from its generalization presented above.
Thus, according to the generalization of the \textit{Tangent Linear System Approximation} the \textit{slow manifold} equation of a
$n$-dimensional dynamical system may be written:

\begin{equation}
\label{eq71}
\phi \left( \vec {X} \right) = \overrightarrow V .\left( {\overrightarrow
{Y_{\lambda _2 } } \wedge \ldots \wedge \overrightarrow {Y_{\lambda _n } } }
\right) = 0
\quad
 \Leftrightarrow
\quad
\overrightarrow V = \sum\limits_{i = 2}^n {a_i \overrightarrow {Y_{\lambda
_i } } } = a_2 \overrightarrow {Y_{\lambda _2 } } + \ldots + a_n
\overrightarrow {Y_{\lambda _n } }
\end{equation}

In the framework of the \textit{Generalized} \textit{Tangent Linear System Approximation} the functional jacobian matrix associated to the
dynamical system has been supposed to be stationary:

\begin{equation}
\label{eq72}
\frac{dJ}{dt} = 0
\end{equation}

As a consequence, time derivatives of acceleration vectors reads:

\[
\mathop {\vec {\gamma }}\limits^{\left( n \right)} = J^{\left( {n +
1} \right)}\overrightarrow V = J^{\left( n \right)}\vec {\gamma }
\]

Then, mapping the \textit{flow} of the \textit{tangent linear system}, i.e., functional jacobian operator $J$ to the
velocity vector field spanned on the eigenbasis (\ref{eq71}) leads to:

\[
J\overrightarrow V = \vec {\gamma } = \sum\limits_{i = 2}^n {a_i
J\overrightarrow {Y_{\lambda _i } } } = a_2 J\overrightarrow {Y_{\lambda _2
} } + \ldots + a_n J\overrightarrow {Y_{\lambda _n } }
\]

\begin{center}
\ldots \ldots \ldots \ldots \ldots \ldots \ldots \ldots \ldots \ldots \ldots
\ldots \ldots \ldots \ldots \ldots \ldots
\end{center}

\[
J^{(n - 2)}\overrightarrow V = \mathop {\vec {\gamma
}}\limits^{\left( {n - 2} \right)} = \sum\limits_{i = 2}^n {a_i
J^{(n - 2)}\overrightarrow {Y_{\lambda _i } } } = a_2 J^{(n -
2)}\overrightarrow {Y_{\lambda _2 } } + \ldots + a_n J^{(n -
2)}\overrightarrow {Y_{\lambda _n } }
\]

While using the eigenequation: $J\overrightarrow {Y_{\lambda _k } } =
\lambda _k \overrightarrow {Y_{\lambda _k } } $ these equations read:

\[
J\overrightarrow V = \vec {\gamma } = \sum\limits_{i = 2}^n {a_i \lambda _i
\overrightarrow {Y_{\lambda _i } } } = a_2 \lambda _2 \overrightarrow
{Y_{\lambda _2 } } + \ldots + a_n \lambda _n \overrightarrow {Y_{\lambda _n
} }
\]

\begin{center}
\ldots \ldots \ldots \ldots \ldots \ldots \ldots \ldots \ldots \ldots \ldots
\ldots \ldots \ldots \ldots \ldots \ldots
\end{center}

\[
J^{(n - 2)}\overrightarrow V = \mathop {\vec {\gamma
}}\limits^{\left( {n - 2} \right)} = \sum\limits_{i = 2}^n {a_i
J^{(n - 2)}\overrightarrow {Y_{\lambda _i } } } = a_2 \lambda _2^{n
- 2} \overrightarrow {Y_{\lambda _2 } } + \ldots + a_n \lambda _n^{n
- 2} \overrightarrow {Y_{\lambda _n } }
\]

Under the assumptions of the \textit{Tangent Linear System
Approximation}, it is obvious that the vectors $\overrightarrow V
$,$\vec {\gamma }$, \ldots , $\mathop {\vec {\gamma
}}\limits^{\left( {n - 2} \right)} $ spanned on the same eigenbasis
$\left( {\overrightarrow {Y_{\lambda _2 } } ,\overrightarrow
{Y_{\lambda _3 } } ,\ldots ,\overrightarrow {Y_{\lambda _n } } }
\right)$ are ``hypercoplanar''. This implies that

\[
\overrightarrow V \cdot \left( {\vec {\gamma } \wedge \dot {\vec
{\gamma }} \wedge \ldots \wedge \mathop {\vec {\gamma
}}\limits^{\left( {n - 2} \right)} } \right) = 0 \quad
 \Leftrightarrow
\quad \dot {\vec {X}} \cdot \left( {\ddot {\vec {X}} \wedge \dddot
{\vec {X}} \wedge \ldots \wedge \mathop {\vec {X}}\limits^{\left( n
\right)} } \right) = 0
\]

Thus, \textit{curvature of the flow} generalizes and encompasses
\textit{Tangent Linear System Approximation}.

\begin{flushright}
$\square$
\end{flushright}

\section{Geometric Singular Perturbation Theory}

Dynamical systems (\ref{eq1}) with small multiplicative parameters in one or several
components of their velocity vector field, i.e., \textit{singularly perturbed systems} may be defined as:

\begin{equation}
\label{eq73}
\left\{ {{\begin{array}{*{20}c}
 {{\vec {x}}' = \vec {f}\left( {\vec {x},\vec {z},\varepsilon } \right)} \\
 {{\vec {z}}' = \varepsilon \vec {g}\left( {\vec {x},\vec {z},\varepsilon }
\right)} \\
\end{array} }} \right.
\end{equation}

where $\vec {x} \in \mathbb{R}^m$, $\vec {z} \in \mathbb{R}^n$, $\varepsilon
\in \mathbb{R}^ + $ and the prime denotes differentiation with respect to
the independent variable $t$. The functions $\vec {f}$ and $\vec {g}$ are
assumed to be $C^\infty $ functions of $\vec {x}$, $\vec {z}$ and
$\varepsilon $ in $U\times I$, where $U$ is an open subset of
$\mathbb{R}^m\times \mathbb{R}^n$ and $I$ is an open interval containing
$\varepsilon = 0$. When $\varepsilon \ll 1$, i.e., is a small positive
number, the variable $\vec {x}$ is called \textit{fast} variable, and $\vec {z}$ is
called \textit{slow} variable. Using Landau's notation: $O\left( {\varepsilon ^k}
\right)$ represents a real polynomial in $\varepsilon $ of $k$ degree, with
$k \in \mathbb{Z}$, it is used to consider that generally $\vec {x}$ evolves
at an $O\left( 1 \right)$ rate; while $\vec {z}$ evolves at an $O\left(
\varepsilon \right)$ \textit{slow} rate. Reformulating the system (\ref{eq1}) in terms of the
rescaled variable $\tau = \varepsilon t$, we obtain:

\begin{equation}
\label{eq74}
\left\{ {{\begin{array}{*{20}c}
 {\varepsilon \dot {\vec {x}} = \vec {f}\left( {\vec {x},\vec
{z},\varepsilon } \right)} \\
 {\dot {\vec {z}} = \vec {g}\left( {\vec {x},\vec {z},\varepsilon } \right)}
\\
\end{array} }} \right.
\end{equation}

The dot $\left( \cdot \right)$ represents as the derivative with
respect to the new independent variable $\tau $. The independent
variables $t$ and $\tau $ are referred to the \textit{fast} and
\textit{slow} times, respectively, and (\ref{eq73}) and (\ref{eq74})
are called \textit{fast} and \textit{slow} system, respectively.
These systems are equivalent whenever $\varepsilon \ne 0$, and they
are labelled \textit{singular perturbation problems} when
$\varepsilon \ll 1$, i.e., is a small positive parameter. The label
\textit{singular} stems in part from the discontinuous limiting
behaviour in the system (\ref{eq73}) as $\varepsilon \to 0^ + $. In
such case, the system (\ref{eq73}) reduces to an $m$-dimensional
system called \textit{reduced fast system}, with the variable $\vec
{z}$ as a constant parameter. System (\ref{eq74}) leads to a
differential-algebraic system called \textit{reduced slow system}
which dimension decreases from $m + n$ to $n. $By exploiting the
decomposition into \textit{fast} and \textit{slow reduced systems}
the geometric approach reduced the full \textit{singularly perturbed
system} to separate lower-dimensional regular perturbation problems
in the \textit{fast} and \textit{slow} regimes, respectively.
\textit{Geometric Singular Perturbation Theory} is based on
Fenichel's assumptions [Fenichel, 1979] recalled below.

\newpage

\subsection{Assumptions}

\hspace{0.16in} \textbf{(H}$_{1}$\textbf{)} The functions $\vec {f}$
and $\vec {g}$ are $C^\infty $ functions in $U\times I$, where $U$
is an open subset of $\mathbb{R}^m\times \mathbb{R}^n$ and $I$ is an
open
interval containing $\varepsilon = 0$.\\

\textbf{(H}$_{2}$\textbf{)} There exists a set $M_0 $ that is
contained in $\{\left( {\vec {x},\vec {z}} \right):\vec {f}\left(
{\vec {x},\vec {z},0} \right) = 0\}$ such that $M_0 $ is a compact
manifold with boundary and $M_0 $ is given by the graph of a $C^1$
function $\vec {x} = \vec {X}_0 \left( \vec {z} \right)$ for $\vec
{z} \in D$, where $D \subseteq \mathbb{R}^n$ is a compact, simply
connected domain and the boundary of $D$ is an $\left( {n - 1}
\right)$ dimensional $C^\infty $ submanifold. Finally, the set $D$
is overflowing invariant with respect to (\ref{eq74}) when
$\varepsilon = 0$.\\

\textbf{(H}$_{2}$\textbf{)} $M_0 $ is normally hyperbolic relative
to the \textit{reduced fast system} and in particular it is required
for all points $\vec {p} \in M_0 $, that there are $k$ (resp. $l)$
eigenvalues of $D_{\vec {x}} \vec {f}\left( {\vec {p},0} \right)$
with positive (resp. negative) real parts bounded away from zero,
where $k + l = m$.

\subsection{Theorems}

\subsubsection*{Fenichel's persistence theorem.}

\textit{Let system (\ref{eq73}) satisfying the conditions (H}$_{1})
-- (H_{3}$\textit{). If }$\varepsilon > 0$\textit{ is sufficiently
small, then there exists a function }$\vec {X}\left( {\vec
{z},\varepsilon } \right)$\textit{ defined on }$D$\textit{ such that
the manifold }$M_\varepsilon = \{\left( {\vec {x},\vec {z}}
\right):\vec {x} = \vec {X}\left( {\vec {z},\varepsilon }
\right)\}$\textit{ is locally invariant under (\ref{eq73}).
Moreover, }$\vec {X}\left( {\vec {z},\varepsilon } \right)$\textit{
is }$C^r$\textit{ for any }$r < + \infty $\textit{, and
}$M_\varepsilon $\textit{ is }$C^rO\left( \varepsilon
\right)$\textit{ close to }$M_0 $\textit{. In addition, there exist
perturbed local stable and unstable manifolds of }$M_\varepsilon
$\textit{. They are unions of invariant families of stable and
unstable fibers of dimensions }$l$\textit{ and }$k$\textit{,
respectively, and they are }$C^rO\left( \varepsilon \right)$\textit{
close for all }$r < + \infty $\textit{ , to their counterparts.}\\

\subsubsection*{Invariance.}

Generally, Fenichel theory enables to turn the problem for explicitly
finding functions $\vec {x} = \vec {X}\left( {\vec {z},\varepsilon }
\right)$ whose graphs are locally \textit{slow invariant manifolds} $M_\varepsilon $ of system (\ref{eq73}) into
regular perturbation problem. Invariance of the manifold $M_\varepsilon $
implies that $\vec {X}\left( {\vec {z},\varepsilon } \right)$ satisfies:

\begin{equation}
\label{eq75} \varepsilon D_{\vec {z}} \vec {X}\left( {\vec
{z},\varepsilon } \right)\vec {g}\left( {\vec {X}\left( {\vec
{z},\varepsilon } \right),\vec {z},\varepsilon } \right) = \vec
{f}\left( {\vec {X}\left( {\vec {z},\varepsilon } \right),\vec
{z},\varepsilon } \right)
\end{equation}

Then, the following perturbation expansion is plugged: $\vec {X}\left( {\vec
{z},\varepsilon } \right) = \vec {X}_0 \left( \vec {z} \right) + \varepsilon
\vec {X}_1 \left( \vec {z} \right) + O\left( {\varepsilon ^2} \right)$ into
(\ref{eq75}) to solve order by order for $\vec {X}\left( {\vec {z},\varepsilon }
\right)$. The Taylor series expansion for $\vec {f}\left( {\vec {X}\left(
{\vec {z},\varepsilon } \right),\vec {z},\varepsilon } \right)$ up to terms
of order two in $\varepsilon $ leads at order $\varepsilon ^0$to

\begin{equation}
\label{eq76}
\vec {f}\left( {\vec {X}_0 \left( {\vec {z},\varepsilon } \right),\vec
{z},0} \right) = \vec {0}
\end{equation}

which defines $\vec {X}_0 \left( \vec {z} \right)$ due to the
invertibility of $D_{\vec {x}} \vec {f}$ and the \textit{implicit
function theorem}. At order $\varepsilon ^1$ we have:

\begin{equation}
\label{eq77} D_{\vec {z}} \vec {X}_0 \left( \vec {z} \right)\vec
{g}\left( {\vec {X}_0 \left( \vec {z} \right),\vec {z},0} \right) =
D_{\vec {x}} \vec {f}\left( {\vec {X}_0 \left( \vec {z} \right),\vec
{z},0} \right)\vec {X}_1 \left( \vec {z} \right) + \frac{\partial
\vec {f}}{\partial \varepsilon }\left( {\vec {X}_0 \left( \vec {z}
\right),\vec {z},0} \right)
\end{equation}

which yields $\vec {X}_1 \left( \vec {z} \right)$ and so forth.

\begin{equation}
\label{eq78} D_{\vec {x}} \vec {f}\left( {\vec {X}_0 \left( \vec {z}
\right),\vec {z},0} \right)\vec {X}_1 \left( \vec {z} \right) =
D_{\vec {z}} \vec {X}_0 \left( \vec {z} \right)\vec {g}\left( {\vec
{X}_0 \left( \vec {z} \right),\vec {z},0} \right) - \frac{\partial
\vec {f}}{\partial \varepsilon }\left( {\vec {X}_0 \left( \vec {z}
\right),\vec {z},0} \right)
\end{equation}

So, regular perturbation theory enables to build locally
\textit{slow invariant} \textit{manifolds} $M_\varepsilon $. But for
high-dimensional \textit{singularly perturbed systems} \textit{slow
invariant manifold} analytical equation determination leads to
tedious calculations.\\

Let's write the \textit{slow invariant manifold} (\ref{eq2}) defined by the \textit{curvature of the flow} as:

\begin{equation}
\label{eq79}
\phi \left( {\vec {x},\vec {z},\varepsilon } \right) = 0
\end{equation}

\begin{proof}

Plugging the perturbation expansion: $\vec {X}\left( {\vec
{z},\varepsilon } \right) = \vec {X}_0 \left( \vec {z} \right) +
\varepsilon \vec {X}_1 \left( \vec {z} \right) + O\left(
{\varepsilon ^2} \right)$ into Eq. (\ref{eq79}) to solve order by
order for $\vec {X}\left( {\vec {z},\varepsilon } \right)$. The
Taylor series expansion for $\phi \left( {\vec {X}\left( {\vec
{z},\varepsilon } \right),\vec {z},\varepsilon } \right)$ up to
terms of suitable order in $\varepsilon $ leads to the same
coefficients as those obtained above.\\

Order $\varepsilon ^0$ provides:

\[
\phi \left( {\vec {X}_0 \left( {\vec {z},\varepsilon } \right),\vec
{z},0} \right) = 0
\]

which also defines $\vec {X}_0 \left( \vec {z} \right)$ due to the
invertibility of $D_{\vec {x}} \vec {f}$ and the \textit{implicit
function theorem}. Thus, \textit{curvature of the flow} encompasses
\textit{Geometric Singular Perturbation Theory}.

\end{proof}

\end{document}